\documentclass[a4paper,11pt]{article}

\pdfoutput=1

\newcommand{\K}[1]{\mathbb{#1}}
\newcommand{\txt}[1]{\text{\normalfont{#1}}}

\usepackage{amsmath}
\usepackage{amsthm}
\usepackage{amssymb}
\usepackage[a4paper, bottom=5.33cm]{geometry}
\usepackage{graphicx}
\usepackage{microtype}
\usepackage{enumitem}
\usepackage{color}
\usepackage{relsize}
\usepackage{mathrsfs}

\usepackage[utf8]{inputenc}
\usepackage[russian,english]{babel}
\usepackage{mathtools} 
\mathtoolsset{showonlyrefs,showmanualtags} 

\numberwithin{equation}{section}

\usepackage[T2A,T1]{fontenc}
\newcommand{\Ch}{\mbox{\usefont{T2A}{\rmdefault}{m}{n}\CYRCH}}
\usepackage{lmodern}

\newcommand{\App}{\mathscr{A}}
\newcommand{\abs}[1]{\left\lvert #1 \right\rvert}
\DeclareMathOperator{\rk}{rank}

\newtheorem{thm}{Theorem}[section]
\newtheorem*{thm*}{Theorem}
\newtheorem{propo}[thm]{Proposition}
\newtheorem{lemma}[thm]{Lemma}
\newtheorem{cor}[thm]{Corollary}

\theoremstyle{definition}
\newtheorem{defi}[thm]{Definition}
\newtheorem{remark}[thm]{Remark}

\makeatletter
\newenvironment{proofof}[1]{\par
  \pushQED{\qed}%
  \normalfont \topsep6\p@\@plus6\p@\relax
  \trivlist
  \item[\hskip\labelsep
    \itshape{Proof of #1\@addpunct{.}}]\ignorespaces
}{%
  \popQED\endtrivlist\@endpefalse
}
\makeatother

\title{Chebyshev polynomials and \\best rank-one approximation ratio}
\author{Andrei Agrachev \and Khazhgali Kozhasov \and Andr\'e Uschmajew}
\date{}

\begin{document}
\maketitle

\begin{changemargin}{1.5cm}{1.5cm}
{\bf\noindent Abstract.}
We establish a new extremal property of the classical Chebyshev polynomials in the context of best rank-one approximation of tensors. We also give some necessary conditions for a tensor to be a minimizer of the ratio of spectral and Frobenius norms.
\end{changemargin}

\section*{Introduction and Outline}

The classical Chebyshev polynomials are known to have many extremal properties. The first result was established by Chebyshev himself: he proved \cite{Chebyshev1854} that a univariate monic polynomial with real coefficients that least deviates from zero on the interval $[-1,1]$ must be proportional to a Chebyshev polynomial of the first kind. Later there were further developments highlighting extremal properties of this class of univariate polynomials and its relevance for approximation theory; see \cite{GONCHAROV20002, OS2008} and references therein. In this article we discover a new extremal property of Chebyshev polynomials of the first kind in the context of the theory of rank-one approximations of real tensors.

Let us define \emph{the binary Chebyshev form of degree $d$} as
\begin{equation}\label{eq:Chebyshev binary}
\Ch_{d,2}(x_1,x_2) = \frac{(x_1+ix_2)^d+(x_1-ix_2)^d}{2} = \sum_{k=0}^{[d/2]} {d \choose 2k} (-1)^kx_1^{d-2k}x_2^{2k}.
\end{equation}
Note that its restriction to the unit circle $x_1^2 + x_2^2 = 1$ can be identified with the univariate Chebyshev polynomial of the first kind $x \mapsto \Ch_{d,2}(x,\sqrt{1-x^2})=\cos(d\arccos x)$, $x\in [-1,1]$. In Theorem \ref{thm:binary} we prove that the binary form \eqref{eq:Chebyshev binary} 
minimizes the ratio of the uniform norm on the unit circle and the Bombieri norm among all nonzero binary forms of the given degree $d$.

In \cite{Qi2011}, the more general problem of minimizing the ratio of the uniform norm on the unit sphere and the Bombieri norm among all nonzero forms of a given degree $d$ and number of variables $n$ was considered. Equivalently, identifying a homogeneous polynomial with the symmetric tensor of its coefficients, one can formulate this problem as follows: minimize the ratio of the spectral norm and the Frobenius norm among all nonzero real symmetric $n^d$-tensors. In an attempt to attack this problem we define the family of homogeneous $n$-ary forms \eqref{eq:Chebyshev form} that we call \emph{Chebyshev forms} $\Ch_{d,n}$.

Besides solving the above problem for the case of binary forms in Theorem \ref{thm:binary}, we solve it in the case of cubic ternary forms ($d=3$, $n=3$) in Theorem \ref{thm:ternary cubics}. This latter result in fact follows from a more general result that we obtain in Theorem \ref{thm:max orthogonal rank}: the maximal orthogonal rank of a real $(3,3,3)$-tensor is $7$. This in particular implies that the minimum value of the ratio of the spectral norm and the Frobenius norm of a nonzero $(3,3,3)$-tensor is $1/\sqrt{7}$ and hence gives an affirmative answer to a conjecture in \cite{KuehnPeetre2006}. 
Since the spectral norm of a tensor measures its relative distance to the set of rank-one tensors (see section~\ref{sub:ratio}) yet another way to interpret this result is that the symmetric tensor associated to the Chebyshev form $\Ch_{3,3}$ achieves the maximum possible relative distance to the set of all rank-one $(3,3,3)$-tensors.


In Theorem \ref{thm:varchar} we show that if a tensor minimizes the ratio of the spectral and the Frobenius norms, then it lies in the space spanned by its best rank-one approximations. In Theorem \ref{thm:varcharsym} we prove an analogous result for symmetric tensors or, equivalently, homogeneous forms: if a form minimizes the ratio of the uniform norm on the unit sphere and the Bombieri norm, then it lies in the space spanned by rank-one forms defined by global extrema of the restriction of the form to the unit sphere. These two results imply lower bounds on the number of best rank-one approximations for those tensors (respectively, on the number of global extrema of homogeneous forms) that minimize the ratio of norms; see Corollary \ref{cor:manybest}. 

In the next section we state all our results in detail. The results are proved in section~\ref{sec:proofs}. Section~\ref{sec: preliminaries} contains some necessary preliminaries and auxiliary results.

\section{Main results}\label{sub:main}

In this section we state our main results. They are all closely related but can be grouped into somewhat different directions.


\subsection{Chebyshev forms and their extremal property}\label{sec: Chebyshev forms}

In the following $P_{d,n}$ denotes the space of real $n$-ary forms of degree $d$ (real homogeneous polynomials of degree $d$ in $n$ variables), and $\| x \| = \sqrt{x_1^2 + \dots + x_n^2}$ is the Euclidean norm on $\K{R}^n$. For a form $p$ we denote by 
\begin{equation*}
\Vert p\Vert_\infty = \max_{\| x \| = 1} \abs{p(x)}
\end{equation*}
the \emph{uniform norm} of its restriction to the unit sphere. 

Every form $p \in P_{d,n}$ has a standard representation in the basis of monomials: $p(x) = \sum_{\abs{\alpha}=d} c_{\alpha}x^\alpha\in P_{d,n}$, where $\alpha = (\alpha_1,\dots,\alpha_n) \in \{0,1,\dots,d\}^n$ is a multi-index of length $|\alpha|=\alpha_1+\dots+\alpha_n=d$ and $x^\alpha = x_1^{\alpha_1} \cdots x_n^{\alpha_n}$.
The \emph{Bombieri norm} \cite{BBEM1990} of $p$ is defined as
\begin{equation}
\Vert p\Vert_B = \sqrt{\sum_{\abs{\alpha}=d} {d\choose \alpha}^{-1} \abs{c_\alpha}^2  },
\end{equation}
where ${d\choose \alpha} = \frac{d!}{\alpha_1!\dots\alpha_n!}$ is the multinomial coefficient.

The \emph{conformal orthogonal group} $CO(n) = \K{R}_+\times O(n)$ acts on the space $P_{d,n}$ of real forms as follows:
\begin{equation}\label{eq: }
  g=(s,\rho)\in CO(n), \  p\in P_{d,n}\ \mapsto\ g^*p \in P_{d,n}, \quad (g^* p)(x) = s p (\rho^{-1}x).
\end{equation}
Note that both the uniform norm and the Bombieri norm are invariant under the subgroup $O(n)$ of orthogonal transformations and their ratio is invariant under the full group $CO(n)$; see section~\ref{sub:ratio}.

In \cite{Qi2011}, Qi asked about the smallest possible ratio $\Vert p \Vert_\infty / \Vert p \Vert_B$ that the two norms can attain in a space $P_{d,n}$. 
In our first result we solve this problem for binary forms of any given degree and we also characterize minimizers in this case.

\begin{thm}\label{thm:binary}
For any nonzero $p\in P_{d,2}$ it holds that
\begin{equation}\label{eq:binary}
  \frac{\,\Vert p\Vert_\infty}{\Vert p\Vert_B} \geq \frac{\,\Vert \Ch_{d,2}\Vert_\infty}{\Vert \Ch_{d,2}\Vert_B} = \frac{1}{\sqrt{2^{d-1}}}.
\end{equation}
When $d=0,1$ one has equality in \eqref{eq:binary} for any $p\in P_{d,2}$, when $d=2$ equality holds if and only if $p=\pm g^*(x_1^2+x_2^2)$ or $p=g^*\Ch_2 = g^*(x_1^2-x_2^2)$, where $g\in CO(2)$. When $d\geq 3$ equality holds if and only if $p=g^*\Ch_{d,2}$, $g\in CO(2)$.
\end{thm}

For any $d\geq 0$ and $n\geq 2$ we define the \emph{$n$-ary Chebyshev form of degree $d$} as 
\begin{equation}\label{eq:Chebyshev form}
  \Ch_{d,n}(x_1,\dots,x_n) = \sum\limits_{k=0}^{[d/2]} {d \choose 2k}(-1)^k x_1^{d-2k}(x_2^2+\dots+x_n^2)^k.
\end{equation}
Note that the forms $\Ch_{d,n}$ are invariant under orthogonal transformations of $\K{R}^{n}$ that preserve the point $(1,0,\dots,0)$ and for any vector $v=(v_2,\dots,v_n)\in \K{R}^{n-1}$ of unit length one has that $\Ch_{d,n}(x_1,v_2x_2,\dots,v_nx_2)=\Ch_{d,2}(x_1,x_2)$ is the binary Chebyshev form~\eqref{eq:Chebyshev binary}. In this work we are particularly concerned with cubic Chebyshev forms
\begin{equation}\label{eq: Chebyshev cubics}
\Ch_{3,n}(x_1,\dots,x_n) = x_1^3 - 3 x_1(x_2^2 + \dots + x_n^2).
\end{equation}

It is an easy calculation that
\begin{equation}\label{eq:norms of Cheb}
\| \Ch_{d,n} \|_\infty = 1, \quad \| \Ch_{d,n} \|_B^2 = \sum\limits_{k=0}^{[d/2]} { d \choose 2k} \sum_{\substack{ \beta = (\beta_1,\dots,\beta_{n-1}), \\ \abs{\beta} = k}} {k \choose \beta}^2 {2k \choose 2\beta}^{-1},
\end{equation}
where $2\beta=(2\beta_1,\dots,2\beta_{n-1})$, and, in particular,
\begin{equation}\label{eq: norm of Chebyshev form}
\| \Ch_{3,n} \|^2_B = 3n - 2.
\end{equation}


In the case $d=2$ of quadratic forms one can easily determine the minimal ratio $\Vert p \Vert_\infty / \Vert p \Vert_B$ by passing to the ratio of spectral and Frobenius norms of symmetric matrices. Specifically,
\begin{equation}
  \frac{\,\Vert p\Vert_{\infty}}{\Vert p\Vert_B}\geq \frac{1}{\sqrt{n}},\quad p\in P_{2,n},
\end{equation}
 with equality only for quadratic forms $p = g^*(\pm x_1^2\pm\dots\pm x_n^2)$, where $g\in CO(n)$. These forms correspond to multiples of symmetric orthogonal matrices. Note that among these extremal quadratic forms there is the Chebyshev quadric $\Ch_{2,n}(x) = x_1^2-x_2^2-\dots-x_n^2$ which is classically known as \emph{the Lorentz quadric}. 
This fact for $d=2$ together with Theorem \ref{thm:binary} might suggest thinking that Chebyshev forms $\Ch_{d,n}$ also minimize the ratio of uniform and Bombieri norm in $P_{d,n}$ for $d\geq 3$ and $n\geq 3$. We show that it is indeed the case for the first ``nontrivial'' situation $d=3$, $n=3$ of ternary cubics.
\begin{thm}\label{thm:ternary cubics}
  Let $p\in P_{3,3}$ be a nonzero ternary cubic form. Then 
  \begin{equation}
    \frac{\,\Vert p\Vert_\infty}{\Vert p\Vert_B} \geq \frac{1}{\sqrt{7}}
  \end{equation}
and equality holds if $p=g^*\Ch_{3,3}$, where $g\in CO(3)$. 
\end{thm}
Theorem~\ref{thm:ternary cubics} is part of Corollary~\ref{cor:App for ternary cubics} further below.


However, at least for all sufficiently large $n$, the Chebyshev form $\Ch_{3,n}$ is not a global minimizer for the norm ratio. Indeed,~\cite[Thm.~5.3]{LZ2019} provides examples of symmetric $n \times n \times n$ tensors with $n = 2^m$ that yield forms $p \in P_{3,2^m}$ satisfying
\begin{equation}\label{eq: counter-example}
\frac{\,\Vert p \Vert_\infty}{\Vert p \Vert_B} = \left( \frac{2}{3} \right)^m = n^{{\frac{\ln (2/3)}{\ln 2}}} \le n^{-0.584},
\end{equation}
whereas, by \eqref{eq:norms of Cheb} and \eqref{eq: norm of Chebyshev form},
\begin{equation}
  \frac{\,\Vert\Ch_{3,n}\Vert_\infty}{\Vert\Ch_{3,n}\Vert_B} = \frac{1}{\sqrt{3n - 2}}.
\end{equation}
For instance, for $n = 2^{10} = 1024$, it holds that $\| \Ch_{3,n} \|_\infty / \| \Ch_{3,n} \|_B \ge 0.0187 > (2/3)^{10} \approx 0.0173$.

Interestingly, while not being a global minimum, one can show that $\Ch_{3,n}$ is a local minimum of the ratio of the two norms on the set of nonzero cubic $n$-ary forms. 
\begin{thm}\label{thm:locmin}
Let $n\geq 2$. For all $p\in P_{3,n}$ in a small neighborhood of $\Ch_{3,n}$ we have
\begin{equation}
 \frac{\,\Vert p\Vert_{\infty}}{\Vert p \Vert_B}\geq  \frac{\,\Vert \Ch_{3,n}\Vert_\infty}{\Vert \Ch_{3,n}\Vert_B}.
\end{equation}
\end{thm}

\subsection{Best rank-one approximation ratio, orthogonal rank, and \\orthogonal~tensors}\label{subsec: best rank-one ratio etc}

Let $\otimes^d_{j=1}\K{R}^{n_j}$ denote the space of real $(n_1,\dots,n_d)$-tensors, considered as $n_1 \times \dots \times n_d$ tables $A=(a_{i_1\dots i_d})$ of real numbers. For two $(n_1,\dots,n_d)$-tensors 
their \emph{Frobenius inner product} is given by
\begin{equation}
  \langle A,A^\prime\rangle_F = \sum_{i_1,\dots,i_d=1}^{n_1,\dots,n_d} a_{i_1\dots i_d}a^\prime_{i_1\dots i_d} 
\end{equation}
and $\| A \|_F = \sqrt{\langle A, A \rangle_F}$ denotes the induced Frobenius norm.

The outer product $x^{(1)} \otimes \dots \otimes x^{(d)}$ of vectors $x^{(j)} \in \K{R}^{n_j}$ is an $(n_1,\dots,n_d)$-tensor $X$ with entries $(x^{(1)}_{i_1} \cdots x^{(d)}_{i_d})$. Nonzero tensors of this form are said to be of \emph{rank one}, denoted $\rk(X) = 1$. 
The \emph{spectral norm} on $\otimes^d_{j=1}\K{R}^{n_j}$ is defined as 
\begin{equation}\label{eq:spectral norm}
  \Vert A\Vert_2 = \max\limits_{\| x^{(1)} \| = \dots = \| x^{(d)} \| = 1} \langle A,x^{(1)}\otimes \dots \otimes x^{(d)}\rangle_F = \max\limits_{\substack{\| X \|_F = 1, \\ \rk(X) = 1}} \langle A,X\rangle_F,
\end{equation} 
where $\| \cdot \|$ denotes the standard Euclidean norm. 

Given a $(n_1,\dots,n_d)$-tensor $A$, a rank-one tensor $Y\in \otimes^d_{j=1}\K{R}^{n_j}$ is called \emph{a best rank-one approximation to $A$} if it minimizes the Frobenius distance to $A$ from the set of rank-one tensors, that is,
\begin{equation}\label{eq:distance}
  \Vert A-Y\Vert_F= \min_{X\in \otimes^d_{j=1}\K{R}^{n_j},\,\txt{rank}(X)=1} \Vert A-X\Vert_F.
\end{equation}
The notion of \emph{best rank-one approximation ratio} of a tensor space was introduced by Qi in \cite{Qi2011}. For the space of $(n_1,\dots,n_d)$-tensors it is defined as
\begin{equation}\label{eq:ratiointro}
  \App(\otimes^d_{j=1}\K{R}^{n_j}) = \min_{0 \neq A\in \otimes^d_{j=1}\K{R}^{n_j}} \frac{\Vert A\Vert_2}{\,\Vert A\Vert_F}.
\end{equation}
It is the largest constant $c$ satisfying $\|A \|_2 \ge c\| A \|_F$ for all $A \in \otimes^d_{j=1}\K{R}^{n_j}$. Another interpretation is that $\App(\otimes^d_{j=1}\K{R}^{n_j})$ is the inverse of the operator norm of the identity map from $(\otimes^d_{j=1}\K{R}^{n_j}, \| \cdot \|_2)$ to $(\otimes^d_{j=1}\K{R}^{n_j}, \| \cdot \|_F)$.

\begin{defi}
A nonzero tensor $A \in \otimes^d_{j=1}\K{R}^{n_j}$ is called \emph{extremal} if it is a minimizer in \eqref{eq:ratiointro}, that is, if it satisfies
\[
\frac{\|A\|_2}{\,\| A \|_F}=\App(\otimes^d_{j=1}\K{R}^{n_j}).
\]
\end{defi}
Seen as a function of a tensor $A\in \otimes^d_{j=1}\K{R}^{n_j}$, $\Vert A\Vert_F=1$, of unit Frobenius norm, the rank-one approximation error \eqref{eq:distance} attains its maximum exactly at extremal tensors of unit Frobenius norm. The precise relation between~\eqref{eq:distance} and~\eqref{eq:ratiointro} together with a possible application is given in~\eqref{eq:relation to rank one} in subsection \ref{sub:ratio}.

The space $\txt{Sym}^d(\K{R}^n)$ of symmetric $n^d$-tensors consists of tensors $A=(a_{i_1\dots i_d})$ in $\otimes^d_{j=1}\K{R}^n$ that satisfy $a_{i_{\sigma_1}\dots i_{\sigma_d}} = a_{i_1\dots i_d}$ for any permutation $\sigma$ on $d$ elements. This space is isomorphic to the space $P_{d,n}$ of homogeneous forms as explained in subsection~\ref{subsec: tensors and forms}. Under this isomorphism Frobenius and spectral norms of a symmetric tensor correspond to Bombieri norm and uniform norm, respectively.  
\emph{The best rank-one approximation ratio $\App(\txt{Sym}^d(\K{R}^n))$} of the space of symmetric tensors is defined by replacing $\otimes^d_{j=1}\K{R}^{n_j}$ with $\txt{Sym}^d(\K{R}^n)$ in~\eqref{eq:ratiointro} and is equal to the minimum ratio between the uniform and the Bombieri norms of a nonzero form in $P_{d,n}$. In this context it is important to note that the definition of the spectral norm of a symmetric tensor does not change if the maximum in~\eqref{eq:spectral norm} is taken over symmetric rank-one tensors only; see subsection~\ref{sub:tensors}.

A general formula for $\App(\otimes^d_{j=1}\K{R}^{n_j})$ or $\App(\txt{Sym}^d(\K{R}^n))$ is not known except for special cases; see~\cite{LNSU2018}. Determining or estimating these constants is an interesting problem on its own and may have some useful applications for rank-truncated tensor optimization methods (see section~\ref{sub:ratio}). The present work contains some new contributions with the main focus on symmetric tensors.

One always has
\begin{equation}
  0<\App(\otimes^d_{j=1}\K{R}^{n_j})\leq 1\quad \txt{and}\quad 0<\App(\otimes^d_{j=1}\K{R}^n)\leq \App(\txt{Sym}^d(\K{R}^n))\leq 1.
\end{equation}
The asymptotic behavior of $\App(\otimes^d_{j=1}\K{R}^n)$ is $O(1/\sqrt{n^{d-1}})$; see~\cite{CobosKuehnPeetre1999}. For $d = 3$ the currently best known upper bound valid for all $n$ seems to be $1.5 n^{\ln(2/3)/\ln 2} \le 1.5 n^{-0.584}$ and follows directly from~\eqref{eq: counter-example}; see~\cite{LZ2019}.

Lower bounds on the best rank-one approximation ratio can be obtained from decomposition of tensors into pairwise orthogonal rank-one tensors. For $A\in \otimes^d_{j=1}\K{R}^{n_j}$ let 
\begin{equation}\label{eq:ortdecintro}
A=Y_1+\dots+Y_r,
\end{equation}
where $Y_1,\dots,Y_r$ are rank-one $(n_1,\dots,n_d)$-tensors such that $\langle Y_\ell,Y_{\ell'}\rangle_F=0$ for $\ell \neq \ell'$. The smallest possible number $r$ that allows such a decomposition \eqref{eq:ortdecintro} is called the \emph{orthogonal rank} of the tensor $A$~\cite{Franc1992} and will be denoted by $\rk_{\perp}(A)$. Since at least one of the terms in~\eqref{eq:ortdecintro} has to satisfy $\langle A, Y_i \rangle_F \ge \| A \|^2_F /r$, it follows that
\begin{equation}
  \frac{\|A\|_2}{\,\|A\|_F} \geq \frac{1}{\sqrt{\rk_{\perp}(A)}}
\end{equation}
for all $A\in \otimes^d_{j=1}\K{R}^{n_j}$. 
Thus an upper bound on the maximal orthogonal rank in a given tensor space leads to a lower bound on the best rank-one approximation ratio of that tensor space:
\begin{equation}\label{eq:lower bound ortho rank}
\App(\otimes^d_{j=1}\K{R}^{n_j}) \ge \frac{1}{\sqrt{\max_{\otimes^d_{j=1}\K{R}^{n_j}} \rk_{\perp}(A)}}.
\end{equation}
It appears that for all known values of $\App(\otimes^d_{j=1}\K{R}^{n_j})$ this is actually an equality~\cite{KM2015,LNSU2018}.

The values for $\App(\K{R}^{n_1}\otimes \K{R}^{n_2}\otimes\K{R}^{n_3})$ have been determined in~\cite{KuehnPeetre2006} for all combinations $n_1,n_2,n_3 \le 4$, except for $(3,3,3)$-tensors. 
In the present work we are able to settle this remaining case, by determining the maximum possible orthogonal rank of a $(3,3,3)$-tensor.
\begin{thm}\label{thm:max orthogonal rank}
The maximal orthogonal rank of a $(3,3,3)$-tensor is seven. 
\end{thm}
In~\cite{KuehnPeetre2006} it has been shown that $1/\sqrt{7}$ is an upper bound for $\App(\K{R}^3\otimes \K{R}^3\otimes\K{R}^3)$ and conjectured that it is actually the exact value. Due to~\eqref{eq:lower bound ortho rank}, Theorem~\ref{thm:max orthogonal rank} shows that $1/\sqrt{7}$ is also a lower bound and hence proves this conjecture. On the other hand, we see from~\eqref{eq:norms of Cheb} and \eqref{eq: norm of Chebyshev form} that the minimal ratio $1/\sqrt{7}$ can be achieved by symmetric $(3,3,3)$-tensors, in particular by the ones associated with the Chebyshev form $\Ch_{3,3}$. Since the spaces $\txt{Sym}^3(\K{R}^3)$ and $P_{3,3}$ are isometric (with respect to the both norms), Theorem~\ref{thm:ternary cubics} is therefore part of the following corollary of Theorem~\ref{thm:max orthogonal rank}.
\begin{cor}\label{cor:App for ternary cubics}
We have 
\begin{equation}
  \App(\K{R}^3\otimes \K{R}^3\otimes\K{R}^3) = \App(\txt{Sym}^3(\K{R}^3)) = \frac{1}{\sqrt{\max_{A\in \K{R}^3\otimes\K{R}^3\otimes\K{R}^3} \rk_{\perp}(A)}} = \frac{1}{\sqrt{7}}
\end{equation}
and the symmetric tensor corresponding to the Chebyshev cubic $\Ch_{3,3}$ is extremal.
\end{cor}

Assume now that $n_1 \le \dots \le n_d$. Then it is not difficult to show that the orthogonal rank of an $(n_1,\dots,n_d)$-tensor is not larger than $n_1\cdots n_{d-1}$. It follows from~\eqref{eq:lower bound ortho rank} that
\begin{equation}\label{eq:naive lower bound}
 \App(\otimes^d_{j=1}\K{R}^{n_j}) \ge \frac{1}{\sqrt{n_1\cdots n_{d-1}}}, \quad n_1 \le \dots \le n_d.
\end{equation}
In~\cite{LNSU2018} the concept of an \emph{orthogonal tensor} is defined by the property that its contraction along the first $d-1$ modes (assuming $n_d$ is the largest dimension) with any $d-1$ vectors of unit length results in a vector of unit length. It is then shown that equality in~\eqref{eq:naive lower bound} is attained if and only if the space contains orthogonal tensors and only those are then the extremal ones. 
Moreover, for $n^d$-tensors this is the case if and only if $n=1,2,4,8$. Therefore, Theorem~\ref{thm:binary} in particular shows that
\[
\App(\txt{Sym}^d(\K{R}^2)) = \frac{1}{\sqrt{2^{d-1}}} = \App(\otimes^d_{j=1}\K{R}^{2}),
\]
and since the symmetric tensors associated to Chebyshev forms attain these constants, they are orthogonal in the sense of~\cite{LNSU2018}. In light of Corollary~\ref{cor:App for ternary cubics} one hence may wonder whether $\App(\txt{Sym}^d(\K{R}^n))$ equals $\App(\otimes^d_{j=1}\K{R}^n)$ in general, or at least in the case $d=3$. Note that this is true for matrices. In general, the answer to this question is, however, negative. 
In the cases $n=4$ and $n=8$ it would imply the existence of symmetric orthogonal tensors, which we show is not possible.
%
%
%
%
%
%
%
%
\begin{propo}\label{propo:n=4,8}
If $A\in \txt{Sym}^d(\K{R}^n)$ is an orthogonal symmetric tensor of order $d\geq 3$, then $n=1$ or $n=2$. For $n=2$ the only such tensors are the ones associated to rotated Chebyshev forms $p=\rho^*\Ch_{d,2}$, $\rho\in O(2)$, that is, are of the form $(\rho,\cdots,\rho) \cdot A$ (see~\eqref{eq:action}) with $A$ given by~\eqref{eq: tensor for Chebyshev}.
\end{propo}
\begin{cor}\label{cor:n=4,8}
For $d\geq 3$ we have
\begin{equation}
 \App(\otimes^d_{j=1} \K{R}^4) = \frac{1}{\sqrt{4^{d-1}}} < \App(\txt{Sym}^d(\K{R}^4))\quad \txt{and}\quad \App(\otimes^d_{j=1} \K{R}^8) = \frac{1}{\sqrt{8^{d-1}}}< \App(\txt{Sym}^d(\K{R}^8)) .
\end{equation}
\end{cor}

The cases of $2^d$- and $(3,3,3)$-tensors are therefore exceptional in the sense that the ``nonsymmetric'' best rank-one approximation ratio can be achieved by symmetric tensors. 

\subsection{Variational characterization and critical tensors}

The problem of determining the best rank-one approximation ratio of a tensor space and finding associated extremal tensors can be seen as a constrained optimization problem for a Lipschitz function. The spectral norm $A \mapsto \| A \|_2$ is a Lipschitz function 
on the normed space $(\otimes^d_{j=1}\K{R}^{n_j},\Vert\cdot\Vert_F)$ (with Lipschitz constant one). The best rank-one approximation ratio $\App(\otimes^d_{j=1}\K{R}^{n_j})$ equals the minimal value of this function on the unit sphere
$\{A\in \otimes^d_{j=1}\K{R}^{n_j}: \Vert A\Vert_F=1\}$ defined by the Frobenius norm,
and extremal tensors (of unit Frobenius norm) are its global minima. Global as well as local minima of a Lipschitz function are among its critical points. The notion of a critical point of a Lipschitz function constrained to a submanifold is explained in section~\ref{sub:gengrad}. It motivates the following terminology. 
\begin{defi}\label{defi:critical}
A nonzero tensor $A\in \otimes^d_{j=1}\K{R}^{n_j}$ is \emph{critical} 
if $A/\Vert A\Vert_F$ is a critical point of the restriction of the spectral norm to the Frobenius sphere, meaning that $\lambda A$ belongs to the generalized gradient of the spectral norm at $A/\Vert A\Vert_F$ for some $\lambda \in \K{R}$. 
\end{defi}
We can then give a characterization of critical $(n_1,\dots,n_d)$-tensors in terms of decompositions of them into their best rank-one approximations.
\begin{thm}\label{thm:varchar}
 A nonzero tensor $A\in \otimes^d_{j=1}\K{R}^{n_j}$ is critical if and only if the rescaled tensor $\Vert A\Vert_2^2/\Vert A\Vert_F^2\, A$ can be written as a convex linear combination of some best rank-one approximations of $A$. 
\end{thm}
Specifically, the theorem states that a tensor $A$ is critical if and only if there exists a decomposition
 \begin{equation}\label{eq:expansion}
  \left(\frac{\Vert A\Vert_2}{\,\Vert A\Vert_F}\right)^2 A = \sum_{\ell = 1}^r \alpha_\ell Y_\ell, \quad  \sum_{\ell=1}^r\alpha_\ell = 1,\ \alpha_1,\dots,\alpha_r>0,
 \end{equation}
 where $Y_1,\dots,Y_r$ are best rank-one approximations of $A$.
In particular, if $A\in \otimes^d_{j=1}\K{R}^{n_j}$ is an extremal tensor, then
\begin{equation}\label{eq:expansion2}
  \App(\otimes^d_{j=1}\K{R}^{n_j})^2 \cdot A = \sum_{\ell = 1}^r \alpha_\ell Y_\ell, \quad  \sum_{\ell=1}^r \alpha_\ell=1,\ \alpha_1,\dots,\alpha_r>0
\end{equation}
for some best rank-one approximations $Y_1,\dots, Y_r$ of $A$.

An analogue of Theorem~\ref{thm:varchar} holds for symmetric tensors or, equivalently,
homogeneous forms. Considering the spectral norm as a function on the space $\txt{Sym}^d(\K{R}^n)$ only, it is again a Lipschitz function, and the best rank-one approximation ratio of $\txt{Sym}^d(\K{R}^n)$ equals its minimum value on the Frobenius unit sphere in the space $\txt{Sym}^d(\K{R}^n)$ of symmetric tensors. 
A nonzero symmetric tensor $A\in \txt{Sym}^d(\K{R}^n)$ is called \emph{critical in $\txt{Sym}^d(\K{R}^n)$} 
if the normalized symmetric tensor $A/\Vert A\Vert_F$ is a critical point (see section~\ref{sub:gengrad}) of the restriction of the spectral norm to the Frobenius sphere in the space $\txt{Sym}^d(\K{R}^n)$. We also say that a form $p \in P_{d,n}$ is critical if the associated symmetric tensor is critical in $\txt{Sym}^d(\K{R}^n)$.
\begin{thm}\label{thm:varcharsym}
  A nonzero tensor $A\in\txt{Sym}^d(\K{R}^n)$ is critical in $\txt{Sym}^d(\K{R}^n)$ if and only if the rescaled tensor $\Vert A\Vert_2^2/\Vert A\Vert_F^2\, A$ can be written as a convex linear combination of some symmetric best rank-one approximations of $A$. In this case $A$ is also critical in the space~$\otimes^d_{j=1}\K{R}^{n}$.
\end{thm}
Here the second statement follows immediately from Theorem~\ref{thm:varchar} and the fact that a best rank-one approximation of a symmetric tensor can always be chosen to be symmetric due to Banach's result \cite{Banach}; see section~\ref{sub:tensors}. However, if $A\in \txt{Sym}^d(\K{R}^n)$ is an extremal symmetric tensor, then, by Theorem \ref{thm:varcharsym},
\begin{equation}\label{eq:expansionsym2}
  \App(\txt{Sym}^d(\K{R}^n))^2 \cdot A = \sum_{\ell = 1}^r \alpha_\ell Y_\ell, \quad  \sum_{\ell=1}^r \alpha_\ell=1,\ \alpha_1,\dots,\alpha_r>0
\end{equation}
for some symmetric best rank-one approximations $Y_1,\dots, Y_r$ of $A$, and $A$ is critical in $\otimes^d_{j=1}\K{R}^{n_j}$. But in general $A$ is not extremal in $\otimes^d_{j=1}\K{R}^{n_j}$ as discussed at the end of the previous subsection. 

Theorems \ref{thm:varchar} and \ref{thm:varcharsym} combined with Proposition \ref{prop:rank} from section \ref{sub:ratio} imply that extremal tensors must have several best rank-one approximations.
\begin{cor}\label{cor:manybest}
Let $d\geq 2$. Then any extremal tensor in $\otimes^d_{j=1}\K{R}^n$ has at least $n$ distinct best rank-one approximations. Similarly, any extremal symmetric tensor in $\txt{Sym}^d(\K{R}^n)$ has at least $n$ distinct symmetric best rank-one approximations. 
\end{cor}

Below we give an alternative characterization of critical tensors in terms of their nuclear norm. The nuclear norm of a $(n_1,\dots,n_d)$-tensor $A\in \otimes^d_{j=1}\K{R}^{n_j}$ is defined by
\begin{equation}\label{eq:prenuclear}
  \Vert A\Vert_* = \inf \left\{\sum\limits_{\ell=1}^r \Vert Y_\ell\Vert_F:\ A=\sum\limits_{\ell=1}^r Y_\ell, \ r\in \K{N},\, \rk(Y_\ell)=1, \ell=1,\dots,r\right\}.
\end{equation}
It is a result of Friedland and Lim~\cite{FriedlandLim} that for a symmetric tensor $A\in \txt{Sym}^d(\K{R}^n)$ it is enough to take the infimum in \eqref{eq:prenuclear} over symmetric rank-one tensors only. 
Hence the nuclear norm of a symmetric tensor can be defined intrinsically in the space $\txt{Sym}^d(\K{R}^n)$. In either case, the infimum in \eqref{eq:prenuclear} is attained.

Nuclear and spectral norms are dual to each other (see subsection \ref{sub:tensors}) and for any tensor $A\in \otimes^d_{j=1}\K{R}^{n_j}$ it holds that
\begin{equation}\label{eq:F2*}
  \Vert A\Vert_F^2\leq \Vert A\Vert_2\Vert A\Vert_*.
\end{equation}
Our next result characterizes tensors achieving equality in \eqref{eq:F2*}.
\begin{thm}\label{thm:criteria}
The following two properties are equivalent for a nonzero tensor $A$ in $\otimes^d_{j=1}\K{R}^{n_j}$ or $\txt{Sym}^d(\K{R}^n)$:
\begin{itemize}
\item[\upshape (i)] $A$ is critical,
\item[\upshape (ii)] $\| A \|_2 \| A \|_* = \| A \|_F^2$.
\end{itemize}
\end{thm}

We remark that the fact that extremal tensors achieve equality in \eqref{eq:F2*} has been already proven in \cite[Theorems~2.2 and~3.1]{DFLW2017}.

\subsection{Decomposition of Chebyshev forms}

For symmetric tensors, the statement of Theorem~\ref{thm:varcharsym} can be reinterpreted in terms of homogeneous forms. Note that a symmetric rank-one tensor $Y = \lambda\, y \otimes \dots \otimes y$, $\lambda\in \K{R}$, $\Vert y\Vert=1$, is a symmetric best rank-one approximation to the symmetric tensor associated to a homogeneous form $p$ if and only if 
\begin{equation}\label{eq: conditions for forms}
\lambda=p(y)=\pm \Vert p\Vert_\infty.
\end{equation}
Also, by~\eqref{eq: power of linear form}, the homogeneous form associated to such a rank-one tensor is proportional to the $d$th power of a linear form,
\begin{equation}
p_{Y}(x) =  \lambda \langle y,x\rangle^d = \lambda (y_1 x_1 + \dots + y_n x_n )^d.
\end{equation}
Therefore, in analogy to~\eqref{eq:expansion}, Theorem~\ref{thm:varcharsym} states that a form $p \in P_{d,n}$ is critical for the ratio $\| p \|_\infty / \| p \|_B$ if and only if it can be written as
\begin{equation}\label{eq: critical decomposition for forms}
\left( \frac{\,\| p \|_\infty}{\| p \|_B} \right)^2 p(x) = \sum_{\ell=1}^r  \alpha_\ell \lambda_\ell \langle y^\ell,x\rangle^d, \quad \sum_{\ell=1}^r \alpha_\ell = 1, \ \alpha_1,\dots,\alpha_r > 0,
\end{equation}
where $\lambda_i\in \K{R}$ and $y^i\in \K{R}^n$, $\Vert y^i\Vert=1$, satisfy ~\eqref{eq: conditions for forms} for $i=1,\dots, r$.

From Theorem \ref{thm:binary} we know that the binary Chebyshev forms $\Ch_{d,2}$ are extremal in $P_{2,d}$ and therefore they must admit a decomposition like~\eqref{eq: critical decomposition for forms}. In Theorem~\ref{prop:decbinary} we provide such a decomposition. For $k=0,\dots, d-1$ denote $\theta_k = \pi k /d$ and $a_k = \cos(\theta_k)$, $b_k = \sin(\theta_k)$. Then $a_k+i b_k= e^{i\theta_k}$, $k=0,\dots, d-1$, are $2d$th roots of unity.

\begin{thm}\label{prop:decbinary}
  For any $d\geq 1$ we have
\begin{equation}\label{eq:dec1}
  \frac{1}{2^{d-1}}\Ch_{d,2}(x_1,x_2) = \frac{1}{d} \sum\limits_{k=0}^{d-1} (-1)^k \left(x_1 a_k + x_2 b_k\right)^d 
\end{equation}
or, in polar coordinates,
\begin{equation}\label{eq:dec2}
 \frac{1}{2^{d-1}}\Ch_{d,2}(\cos\theta,\sin\theta)=  \frac{1}{2^{d-1}}\cos(d\theta) = \frac{1}{d}\sum_{k=0}^{d-1} (-1)^k\cos(\theta-\theta_k)^d.
\end{equation}
\end{thm}

The second equality in~\eqref{eq:dec2} constitutes an interesting trigonometric identity, which we were not able to find in the literature.

In the following corollary of Theorem~\ref{prop:decbinary} we provide a decomposition~\eqref{eq: critical decomposition for forms} for cubic Chebyshev forms $\Ch_{3,n}$, which shows that they are critical in $P_{3,n}$.

\begin{cor}\label{cor:deccubic} For $n\geq 2$ we have
  \begin{equation}\label{eq:Ch3n}
    \frac{1}{3n-2}\Ch_{3,n}(x) = \left(\frac{n+2}{9n-6}\right) x_1^3 + \frac{4}{9n-6}\sum_{i=2}^n -\left( \frac{x_1 + \sqrt{3}x_i}{2}\right)^3 +  \left(\frac{-x_1 +\sqrt{3}x_i}{2}\right)^3.
  \end{equation}
In particular, $\Ch_{3,n}$, $n\geq 2$, is critical for the ratio $\Vert p \Vert_\infty/\Vert p\Vert_B$, $p\in P_{3,n}$.
\end{cor}

In section \ref{sub:localoptimality} we use Corollary \ref{cor:deccubic} to prove Theorem \ref{thm:locmin}, that is, that $\Ch_{3,n}$ is a local minimum for the norm ratio.

It is interesting to note that a decomposition of $\Ch_{3,n}$ or, more precisely, of its representing symmetric tensor, into nonsymmetric best rank-one approximations is trivially obtained. By~\eqref{eq: Chebyshev cubics}, the associated symmetric tensor is
\begin{equation}\label{eq:non-symmetric decomposition}
A_n = e_1 \otimes e_1 \otimes e_1 - \sum_{k=2}^n (e_1 \otimes e_k \otimes e_k + e_k \otimes e_1 \otimes e_k + e_k \otimes e_k \otimes e_1),
\end{equation}
where $e_1,\dots, e_n$ denote the basic unit vectors in $\K{R}^n$. Since $\|A_n \|_2 = 1$ by~\eqref{eq: norm of Chebyshev form}, this ``decomposition into entries'' is a decomposition into best rank-one approximations with equal weights. Scaling by $\| A_n \|^2_2 / \| A_n \|_F^2 = 1/(3n-2)$ provides a desired convex decomposition~\eqref{eq:expansion}. While this proves that $A_n$ is critical in $\K{R}^n \otimes \K{R}^n \otimes \K{R}^n$ (see Theorem~\ref{thm:varchar}), it does not imply by itself that $A_n$ is critical in $\txt{Sym}^3(\K{R}^n)$. Thus Corollary~\ref{cor:deccubic} is a stronger statement. Observe also that~\eqref{eq:non-symmetric decomposition} is a decomposition into pairwise orthogonal rank-one tensors. This together with \eqref{eq: norm of Chebyshev form} and \eqref{eq:lower bound ortho rank} shows that the tensor $A_n$ associated with the cubic Chebyshev form $\Ch_{3,n}$ has orthogonal rank $3n - 2$.

\section{Preliminaries}\label{sec: preliminaries}

In this section we gather some basic definitions and preliminary results upon which we base our arguments for proving the main results in section~\ref{sec:proofs}.

\subsection{Tensors, forms, and their norms}\label{sub:tensors}\label{subsec: tensors and forms}


The space of $(n_1,\dots,n_d)$-tensors is isomorphic to the space of multilinear maps on $\K{R}^{n_1}\times\dots\times \K{R}^{n_d}$. The map associated to a tensor $A$ is given by
\begin{equation}\label{eq:isomorphy with multilinear form}
(x^{(1)},\dots,x^{(d)}) \mapsto  \langle A, x^{(1)}\otimes \dots\otimes x^{(d)}\rangle_F= \sum\limits_{i_1,\dots,i_d=1}^{n_1,\dots,n_d} a_{i_1\dots i_d} x^{(1)}_{i_1} \dots x^{(d)}_{i_d}. 
\end{equation}
The spectral norm \eqref{eq:spectral norm} of $A$
equals the uniform norm of the restriction of the associated multilinear map to the product of unit spheres in $\K{R}^{n_1}\times\dots\times\K{R}^{n_d}$.

As for the nuclear norm defined in~\eqref{eq:prenuclear}, it can be shown that the infimum is always attained (see \cite[Prop. $3.1$]{FriedlandLim})
and a decomposition $A=\sum_{\ell=1}^r X_\ell$ of $A$ into rank-one tensors such that $\Vert A\Vert_* = \sum_{\ell=1}^r \Vert X_\ell\Vert_F$ is called \emph{a nuclear decomposition}. 
We have already stated that 
the spectral and the nuclear norms are dual to each other, that is,
\begin{equation}\label{eq:duality}
\Vert A\Vert_2 = \max\limits_{\Vert A^{\prime}\Vert_*\leq 1} \abs{\langle A,A^{\prime}\rangle_F},\quad
  \Vert A\Vert_* = \max\limits_{\Vert A^{\prime}\Vert_2\leq 1} \abs{\langle A,A^{\prime}\rangle_F},
\end{equation}
and the three above introduced norms satisfy
\begin{equation}\label{eq:relations}
  \Vert A\Vert_2\leq \Vert A\Vert_F, \quad \Vert A\Vert_F\leq \Vert A\Vert_*,\quad \txt{and}\quad \Vert A\Vert_F^2\leq \Vert A\Vert_2\Vert A\Vert_*.
\end{equation}
Moreover, in the first two inequalities in \eqref{eq:relations} equality holds if and only if $A$ is a rank-one tensor. We refer to~\cite{CKP1992, FriedlandLim} for these statements.

The product of orthogonal groups \mbox{$O(n_1,\dots,n_d) = O(n_1)\times\dots\times O(n_d)$}, whose elements are denoted $(\rho^{(1)},\dots,\rho^{(d)})$, acts on the space $\otimes^d_{j=1}\K{R}^{n_j}$ as
\begin{equation}\label{eq:action}
(\rho^{(1)},\dots,\rho^{(d)}) \cdot A = \left(\sum\limits_{j_1,\dots,j_d=1}^{n_1,\dots,n_d} \rho^{(1)}_{i_1j_1}\dots\rho^{(d)}_{i_dj_d} a_{j_1\dots j_d}\right),
\end{equation}
preserving the Frobenius inner product and the spectral and the nuclear norms. 

The ${n+d-1 \choose d}$-dimensional space $\txt{Sym}^d(\K{R}^n)\subset \otimes^d_{j=1}\K{R}^n$ of symmetric $n^d$-tensors is isomorphic to the space $P_{d,n}$ of $n$-ary $d$-homogeneous real forms. The symmetric tensor $A$ is identified with the form $p_A$ defined as
\begin{equation}\label{eq:symtensors-forms}
p_A(x)= \langle A,x\otimes \dots\otimes x\rangle_F= \sum\limits_{i_1,\dots,i_d=1}^n a_{i_1\dots i_d}x_{i_1}\dots x_{i_d},\quad x\in \K{R}^n,
\end{equation}
which equals the restriction of the multilinear map~\eqref{eq:isomorphy with multilinear form} to the ``diagonal'' in $\K{R}^n \times \dots \times \K{R}^n$. 
It is convenient to represent $p_A$ in the basis of monomials 
\[
p_A(x) = \sum_{\abs{\alpha}=d} a_{\alpha}x^{\alpha},
\]
where
\begin{equation}\label{eq:definition of coefficient}
a_{\alpha} = {d\choose \alpha} a_{i_1\dots i_d}
\end{equation}
and $\{i_1,\dots,i_d\}$ is any collection of indices such that for $i=1,\dots, n$ the value $i$ occurs $\alpha_i$ times among $i_1,\dots,i_d$.

As an example, the binary Chebyshev form $\Ch_{d,2}$ in~\eqref{eq:Chebyshev binary} corresponds to the symmetric tensor with entries
\begin{equation}\label{eq: tensor for Chebyshev}
 a_{i_1\dots i_d} = \begin{cases}
                     (-1)^k \quad &\txt{if $\#\{i_j=2\} = 2k$,} \\
                     0 \quad &\txt{otherwise}
                    \end{cases}
\end{equation}
and the associated multilinear map~\eqref{eq:isomorphy with multilinear form} is given by
\begin{equation}\label{eq: chebyshev desymmetrized}
 \langle A, x^{(1)} \otimes \dots \otimes x^{(d)} \rangle = \sum\limits_{k=0}^{[d/2]} (-1)^k\sum\limits_{\#\{i_j=2\} = 2k} x^{(1)}_{i_1}\dots x^{(d)}_{i_{d}}.
\end{equation}

%

Banach proved~\cite{Banach} that for a symmetric coefficient tensor $A$, the maximum absolute value of the multilinear form~\eqref{eq:isomorphy with multilinear form} on a product of spheres can be attained at diagonal inputs, in other words,
\begin{equation}\label{eq:spectral=uniform}
  \Vert A\Vert_2 = \max\limits_{\| x \| = 1} \abs{p_A(x)} = \Vert p_A\Vert_{\infty}.
\end{equation}
This is a generalization of the fact that for a symmetric matrix $A$ the maximum absolute value of the bilinear form $x^T A y$ is, modulo scaling, attained when $x=y$ is an eigenvector for the eigenvalue with the largest absolute value. Therefore, spectral norm for symmetric tensors may be intrinsically defined in the space $\txt{Sym}^d(\K{R}^n)$. 


Next, one can easily check that the Frobenius inner product between two symmetric tensors $A=(a_{i_1\dots i_d})$, $A^{\prime}=(a^{\prime}_{i_1\dots i_d})\in \txt{Sym}^d(\K{R}^n)$ equals the \emph{Bombieri product} between the corresponding homogeneous forms $p_A(x)=\sum_{\abs{\alpha}=d} a_{\alpha}x^{\alpha}$ and $p_{A^{\prime}}(x)=\sum_{\abs{\alpha}=d} a^{\prime}_{\alpha} x^{\alpha}$ with coefficients defined through~\eqref{eq:definition of coefficient}:
\begin{equation}\label{eq:F=B}
  \langle A, A^{\prime}\rangle_{F} = \sum\limits_{i_1,\dots,i_d=1}^n a^{}_{i_1\dots i_d}a^{\prime}_{i_1\dots i_d} = \sum\limits_{\abs{\alpha}=d} {d\choose \alpha}^{-1} a^{}_{\alpha} a^{\prime}_{\alpha} \eqqcolon \langle p_A, p_{A^{\prime}}\rangle_{B}.
\end{equation}

By~\eqref{eq:spectral=uniform} and~\eqref{eq:F=B}, the isomorphism $A \mapsto p_A$ establishes an isometry between $(\txt{Sym}^d(\K{R}^n), \| \cdot \|_2)$ and $(P_{d,n},\| \cdot \|_\infty)$, as well as between $(\txt{Sym}^d(\K{R}^n), \| \cdot \|_F)$ and $(P_{d,n},\| \cdot \|_B)$. 

When $n_1=\dots=n_d=n$ the diagonal subaction of the action \eqref{eq:action} preserves the subspace $\txt{Sym}^d(\K{R}^n)$ of symmetric tensors and it corresponds to the action of the orthogonal group on the space $P_{d,n}$ of homogeneous forms by orthogonal change of variables:
\begin{equation}\label{eq:action2}
  \rho\in O(n), \ p\in P_{d,n} \ \mapsto\ \rho^*p \in P_{d,n}, \quad (\rho^* p)(x) = p (\rho^{-1}x).
\end{equation}
Due to \eqref{eq:F=B}, this shows that the Bombieri inner product is invariant under such a change of variables.


Finally, we have already noted that according to~\eqref{eq:symtensors-forms} a symmetric rank-one tensor $Y = \pm  y \otimes \cdots \otimes y$ corresponds to the $d$th power of a linear form $\langle y,\cdot \rangle$ as follows:
\begin{equation}\label{eq: power of linear form}
p_{Y}(x) = \langle \pm y \otimes \dots \otimes y, x \otimes \dots \otimes x \rangle_F = \pm \langle y, x \rangle^d.
\end{equation}
Hence a decomposition of a symmetric tensor into symmetric rank-one tensors corresponds to a decomposition of the associated homogeneous form into powers of linear forms. Note that by~\eqref{eq:symtensors-forms} the Bombieri inner product of any homogeneous form $p \in P_{d,n}$ with a $d$th power of a linear form $\langle y,\cdot\rangle$ equals
\(
\langle p , \langle y,\cdot\rangle^d \rangle_B = p(y).
\)

\subsection{Best rank-one approximation ratio}\label{sub:ratio}


Given a nonzero tensor $A\in \otimes^d_{j=1}\K{R}^{n_j}$, a rank-one $(n_1,\dots,n_d)$-tensor $Y=\lambda\, y^{(1)}\otimes\dots\otimes y^{(d)}$, where $\lambda\in \K{R}$ and $\Vert y^{(i)}\Vert = 1$, $i=1,\dots, d$, is a best rank-one approximation to $A$ if and only if
\begin{equation}\label{eq:app}
  \lambda = \langle A,y^{(1)}\otimes\dots\otimes y^{(d)}\rangle_F = \pm \Vert A\Vert_2.
\end{equation}
Banach's result \cite{Banach} implies that one can take $y^{(1)}=\dots=y^{(d)}\in \K{R}^n$ in \eqref{eq:app} if the tensor $A\in \txt{Sym}^d(\K{R}^n)$ is symmetric. 


Also if $Y=\lambda\,y^{(1)} \otimes \dots \otimes y^{(d)}$ is a best rank-one approximation of $A$ as above, then for every $j=1,\dots,d$ the linear form $x^{(j)} \mapsto \langle A, y^{(1)} \otimes \cdots \otimes x^{(j)} \otimes \cdots \otimes y^{(d)} \rangle_F$ constrained to $\| x^{(j)} \| = 1$ achieves its maximum at $y^{(j)}$ and hence it vanishes on the orthogonal complement of $y^{(j)}$, that is,
\begin{equation}\label{lemma:normal form}
  \langle A, y^{(1)}\otimes \dots \otimes y^{(j-1)}\otimes x^{(j)}\otimes y^{(j+1)} \otimes \dots\otimes y^{(d)}\rangle_F = 0
\end{equation}
for all $x^{(j)}\in \K{R}^{n_j}$ that are orthogonal to $y^{(j)}$.

We continue with some remarks on extremal tensors and best rank-one approximation ratio. From the definition \eqref{eq:distance} of a best rank-one approximation and \eqref{eq:app} we have
\begin{equation}\label{eq: distance to rank one}
\min_{\txt{rank}(X)=1} \Vert A-X\Vert_F^2 = \| A - Y \|_F^2 = \| A \|_F^2 - \| A \|_2^2
\end{equation}
for any best rank-one approximation $Y$ to $A\in \otimes^d_{j=1}\K{R}^{n_j}$.
Recalling the definition~\eqref{eq:ratiointro} of the best rank-one approximation ratio $\App(\otimes^d_{j=1}\K{R}^{n_j})$, the maximum relative distance of a tensor to the set of rank-one tensors is given as
\begin{equation}\label{eq:relation to rank one}
\max_{0\neq A\in \otimes^d_{j=1}\K{R}^{n_j}} \min_{\rk(X) = 1} \frac{\| A - X \|_F}{\|A\|_F} 
= \sqrt{1 - \App(\otimes^d_{j=1}\K{R}^{n_j})^2}
\end{equation}
and is achieved for extremal tensors. This relation explains the name ``best rank-one approximation ratio'' for the constant $\App(\otimes^d_{j=1}\K{R}^{n_j})$. When restricting to symmetric tensors,~\eqref{eq:relation to rank one} holds with $\App(\txt{Sym}^d(\K{R}^n))$ instead.

For context we note that the relation~\eqref{eq:relation to rank one} shows that lower bounds on $\App(\otimes^d_{j=1}\K{R}^{n_j})$ can be used for convergence analysis of greedy methods for low-rank approximation using rank-one tensors as a dictionary. For example, the pure greedy method to approximate $A\in \otimes^d_{j=1}\K{R}^{n_j}$ produces a recursive sequence $A_{\ell+1} = A_\ell + Y_\ell$, where $A_0 = 0$ and $Y_\ell$ is a best rank-one approximation of $A - A_\ell$. Then~\eqref{eq:relation to rank one} implies
\[
\| A - A_{\ell+1} \|_F \le \sqrt{1 - \App(\otimes^d_{j=1}\K{R}^{n_j})^2} \| A - A_\ell \|_F;
\]
see~\cite{Temlyakov:2011} for a general introduction to greedy methods. For a more general problem of finding an approximate low-rank minimizer for a smooth cost function $f \colon \otimes^d_{j=1}\K{R}^{n_j} \to \K{R}$, one could replace $Y_\ell$ with a (scaled) best rank-one approximation of a suitable residual, for example, the negative gradient $-\nabla f(A_\ell)$. Then $\App(\otimes^d_{j=1}\K{R}^{n_j})$ is a lower bound for the (cosine of the) angle between the search direction $Y_\ell$ and $-\nabla f(A_\ell)$ and hence can be used to estimate the convergence of such an iteration; see, e.g.,~\cite{U2015} and references therein. Again, for symmetric tensors once can replace $\App(\otimes^d_{j=1}\K{R}^{n_j})$ with $\App(\txt{Sym}^d(\K{R}^n))$ in these considerations.

In the following lemma we show that the best rank-one approximation ratio strictly decreases with the dimension.
\begin{lemma}\label{lem:monotone}
Let $\App_{d,n}$ denote either $\App(\otimes^d_{j=1}\K{R}^{n})$ or $\App(\txt{Sym}^d(\K{R}^{n}))$. Then for any $d\geq 1$ and $n\geq 1$ we have
\[
\App_{d,n+1} \le \frac{\App_{d,n}}{\sqrt{1 + \App_{d,n}^2}}. 
\]
\end{lemma}
\begin{proof}
Let $A \in \otimes^d_{j=1}\K{R}^{n}$ be an $n^d$-tensor of unit Frobenius norm, $\|A \|_F = 1$. For $\varepsilon\in [0,1]$, let $A^\varepsilon \in \otimes^d_{j=1}\K{R}^{n+1}$ be the $(n+1)^d$-tensor with entries
\begin{equation}\label{eq:defA}
  a^{\varepsilon}_{i_1\dots i_d}=
  \begin{cases}
    \sqrt{1-\varepsilon^2\|A \|_2^2} a_{i_1\dots i_d} &\txt{if $i_1,\dots,i_d \le n$,}\\
    \varepsilon \| A \|_2 &\txt{if $i_1=\dots=i_d=n+1$,}\\
    0 &\txt{otherwise.}
\end{cases}
\end{equation}
Observe that $\| A^\varepsilon \|_F =1$, and $A^\varepsilon$ is symmetric if $A$ is. Let $\xi^{(1)},\dots,\xi^{(d)}$ be unit norm vectors in $\K{R}^{n+1}$ partitioned as $\xi^{(j)} = (x^{(j)},z^{(j)})$ with $x^{(j)} \in \K{R}^{n}$ and $z^{(j)} \in \K{R}$. Then from the ``block diagonal'' structure of $A^\varepsilon$ it follows that
\begin{align*}
\langle A^\varepsilon, \xi^{(1)} \otimes \dots \otimes \xi^{(d)} \rangle_F &= \sqrt{1-\varepsilon^2\| A \|_2^2} \langle  A, x^{(1)} \otimes \dots \otimes x^{(d)} \rangle_F +  \varepsilon\| A \|_2 z^{(1)} \cdots z^{(d)} \\
&\le \max\left(\sqrt{1-\varepsilon^2\|A\|_2^2}, \varepsilon \right) \| A \|_2^{} (\| x^{(1)} \| \cdots \| x^{(d)} \| + z^{(1)} \cdots z^{(d)}).
\end{align*}
By a generalized H\"older inequality \cite[\S~11]{Hardy}, the term in the right brackets is bounded by one. The maximum on the left, on the other hand, takes its minimal value for $\varepsilon = 1/ \sqrt{1 + \| A \|_2^2}$. Since $\xi^{(1)}, \cdots, \xi^{(d)}$ were arbitrary, this shows
\[
\|A ^\varepsilon \|_2 \le \frac{\| A \|_2}{\sqrt{1 + \| A \|_2^2}} .
\]
The assertions follow by choosing $A$ to be an extremal tensor in the space $\otimes^d_{j=1}\K{R}^{n}$ or $\txt{Sym}^d(\K{R}^{n})$, respectively.
\end{proof}
The previous lemma provides a lower bound on the rank of extremal tensors. 
Recall that the \emph{(real) rank} of a tensor $A\in \otimes^d_{j=1}\K{R}^{n_j}$ is the smallest number $r$ that is needed to represent $A$ as the linear combination
\begin{equation}\label{eq:rank}
A = X_1 + \dots + X_r
\end{equation}
of rank-one tensors $X_1,\dots,X_r$. 
The \emph{(real) symmetric rank} of a symmetric tensor $A$ is the smallest number of symmetric rank-one tensors needed for~\eqref{eq:rank} to hold.
\begin{propo}\label{prop:rank}
If $A\in \otimes^d_{j=1}\K{R}^n$ is an extremal tensor, its rank must be at least $n$.
If $A\in \txt{Sym}^d(\K{R}^n)$ is an extremal symmetric tensor, its symmetric rank must be at least $n$. 
\end{propo}
\begin{proof}
Let  $A\in \otimes^d_{j=1}\K{R}^n$ be a tensor of rank at most $n-1$, that is,
\begin{equation}
  A=v^{(1)}_1\otimes\dots\otimes v^{(d)}_1+\dots+v^{(1)}_{n-1}\otimes\dots\otimes v^{(d)}_{n-1}.
\end{equation}
For $j=1,\dots,d$ let $V^{(j)}\simeq \K{R}^{n-1}$ be any $(n-1)$-dimensional subspace of $\K{R}^n$ that contains vectors $v^{(j)}_1,\dots, v^{(j)}_{n-1}$. Since $A\in V^{(1)}\otimes \dots\otimes V^{(d)}\simeq \otimes^d_{j=1}\K{R}^{n-1}$ we have 
\begin{equation}
\frac{\Vert A\Vert_2}{\,\Vert A\Vert_F} \geq  \App(\otimes_{j=1}^d\K{R}^{n-1}).
\end{equation}
Thus, by Lemma~\ref{lem:monotone}, $A$ cannot be extremal in $\otimes_{j=1}^d\K{R}^{n}$. 

When $A$ is symmetric and of symmetric rank at most $n-1$ we can choose $V^{(1)}=\dots=V^{(d)} = V$ so that $A\in \txt{Sym}^d(V)\simeq \txt{Sym}^d(\K{R}^{n-1})$, leading to the analogous conclusion.
%
\end{proof}

\subsection{Generalized gradients and local optimality of Lipschitz functions }\label{sub:gengrad}

The problem of determining the best rank-one approximation ratio of a tensor space and finding extremal tensors is a constrained optimization problem for a Lipschitz function.
 The theory of generalized gradients developed by Clarke~\cite{Clarke1975} provides necessary optimality conditions. 
We provide here only the most necessary facts of this theory needed for our results. A comprehensive introduction is given, e.g., in~\cite{ClarkeBook1990}.
 
A function $f: \K{R}^m \rightarrow \K{R}$ is called \emph{Lipschitz}, if there exist a constant $L$ such that $\abs{f(p)-f(q)}\leq L\Vert p-q\Vert$ for all pairs $p, q\in \K{R}^m$. By the classical Rademacher's theorem, a Lipschitz function $f$ is differentiable at almost all (in the sense of Lebesgue measure) points $p\in \K{R}^m$. Denote by $\nabla f(p)$ the gradient of $f$ at such a point. The \emph{generalized gradient} of $f$ at any $p \in \K{R}^m$, denoted as $\partial  f(p)$, is then defined as the convex hull of the set of all limits $\nabla f(p_i)$, where $p_i$ is a sequence of differentiable points that converges to $p$. It turns out that $\partial f(p)$ is a nonempty convex compact subset of $\K{R}^m$. Moreover $\partial f(p)$ is a singleton if and only if $f$ is differentiable at $p$, in which case $\partial f(p) = \{ \nabla f(p) \}$.
 
 

Let $S$ be a differentiable submanifold in $\K{R}^m$. Then a necessary condition for the Lipschitz function $f$ to attain a local minimum relative to $S$ at $x \in S$ is that
\begin{equation}\label{eq:condition}
\partial f(p) \cap N_S(p) \neq \emptyset,
\end{equation}
where $N_S(p)$ denotes the normal space, that is, the orthogonal complement of the tangent space of $S$ at $p$. Note that this is a ``Lipschitz'' analogue of the classical Lagrange multipliers rule for continuously differentiable functions. We refer to~\cite[Sec.~2.4]{ClarkeBook1990}. Every point $p \in S$ that satisfies~\eqref{eq:condition} is called a \emph{critical point} of $f$ on $S$. Hence local minima of $f$ on $S$ are among the critical points.




The proofs of Theorems~\ref{thm:varchar} and~\ref{thm:varcharsym} in section \ref{sub:varchar} consist in applying the necessary optimality condition~\eqref{eq:condition} to the spectral norm function on the sphere defined by Frobenius norm. Here two things are of relevance. First, for a Euclidean sphere $S$ we have $N_s(p) = \{\mu p: \mu \in \K{R}\}$. Hence the condition~\eqref{eq:condition} becomes
\begin{equation}\label{eq:nec condition on sphere}
\mu p \in \partial f(p)
\end{equation}
for some $\mu \in \K{R}$. Second, by~\eqref{eq:spectral norm}, the spectral norm is an example of a so-called max function, that is, a function of the type
\begin{equation}\label{eq:max function}
f(p) = \max_{u \in C} g(p,u),
\end{equation}
where $C$ is compact. Under certain smoothness conditions on the function $g$, which are satisfied for spectral norm~\eqref{eq:spectral norm}, Clarke~\cite[Thm.~2.1]{Clarke1975} has determined the following characterization of the generalized gradient:
\begin{equation}\label{eq:general formula for subdifferential}
\partial f(p) = \txt{conv} \{ \nabla_p g(p,u) : u \in M(p) \},
\end{equation}
where $\text{conv}$ denotes the convex hull and $M(p)$ is the set of all maximizers $u$ in~\eqref{eq:max function} for a fixed $p$. For the spectral norm~\eqref{eq:spectral norm}, this set consists of all normalized best rank-one approximations of a given tensor; see~\eqref{eq:subdifferential of spectral norm}.


\section{Proof of main results}\label{sec:proofs}

Our main results are proved in this section. We are going to repeatedly use the equivalence \eqref{eq:symtensors-forms} between symmetric tensors and homogeneous forms and the corresponding relations \eqref{eq:spectral=uniform}, \eqref{eq:F=B} for the different norms.

\subsection{Binary forms}\label{sub:binary forms}

This subsection is devoted to the proof of Theorem \ref{thm:binary}. 
While the given proof is self-contained, some arguments could be omitted with reference to results in~\cite{LNSU2018}.

\begin{proofof}{Theorem \ref{thm:binary}}
By~\eqref{eq:norms of Cheb},
\[
\frac{\| \Ch_{d,2} \|_2}{\, \| \Ch_{d,2} \|_B} = \frac{1}{\sqrt{2^{d-1}}}.
\]
It then follows from~\eqref{eq:naive lower bound} that this value equals $\App(\otimes^d_{j=1}\K{R}^{2})$, so the symmetric tensor associated to the Chebyshev form must be extremal both in $\otimes^d_{j=1}\K{R}^{2}$ and in $\txt{Sym}^d(\K{R}^2)$.

We now consider the uniqueness statements. When $d=1$, the space $P_{1,n}$ consists of linear forms $p(x) = \langle a, x\rangle$, for any of which it holds that $\| p \|_\infty / \| p \|_B = 1$. 
In the case $d=2$ of quadratic forms, the minimal ratio between spectral and Frobenius norm of a symmetric $n \times n$ matrix is attained for multiples of symmetric orthogonal matrices only and takes the value $1/\sqrt{n}$. When $n=2$, all such matrices can be obtained by orthogonal transformation and scaling from the two diagonal matrices with diagonal entries $(1,1)$ and $(1,-1)$, respectively. This corresponds to the asserted quadratic forms $p \in P_{d,2}$.

In the case $d \ge 3$ we have to show that the only symmetric $2^d$-tensors $A$ satisfying
\begin{equation}\label{eq:optimal values}
\| A \|_2 = 1, \quad \| A \|_F = \sqrt{2^{d-1}}
\end{equation}
are obtained from orthogonal transformations of the Chebyshev form $\Ch_{d,2}$. To this end, we show that under the additional condition
\begin{equation}\label{eq:conditions1}
 p_A(e_1) = \langle A, e_1 \otimes \dots \otimes e_1 \rangle_F = 1 = \| A \|_2,
\end{equation}
the form $p_A$ equals $\Ch_{d,2}$. The proof is given by induction over $d \ge 3$. Before giving this proof we note that for a $2^d$-tensor $A$ satisfying~\eqref{eq:optimal values}, its two \emph{slices} $A_1 = (a_{i_1\dots i_{d-2} 1})$ and $A_2 = (a_{i_1\dots i_{d-1} 2})$ necessarily have the same Frobenius norm $\| A_1 \|_F = \| A_2 \|_F = \sqrt{2^{d-2}}$. In fact, $\| A \|_2 = 1$ implies $\| A_1 \|_2 \le 1$ and hence, by~\eqref{eq:naive lower bound}, $\| A_1 \|_F \le \sqrt{2^{d-2}}$. Since the same holds for $A_2$ and since $\| A \|_F^2 = \| A_1 \|_F^2 + \| A_2 \|_F^2$, the claim follows. Moreover, $\| A_1 \|_2 = \| A_2 \|_2 = 1$, again by~\eqref{eq:naive lower bound}, so that both slices are necessarily extremal. Note that by the same argument, every $2^{d'}$-subtensor of $A$ with $d' < d$ must be extremal.

We begin the induction with $d=3$. Assume $A\in \txt{Sym}^3(\K{R}^2)$ satisfies~\eqref{eq:optimal values} and~\eqref{eq:conditions1}. Then we have seen that both, say, frontal slices of $A$ are themselves extremal symmetric $2 \times 2$ matrices. By~\eqref{eq:conditions1}, $a_{111} = 1$ and the tensor $e_1 \otimes e_1 \otimes e_1$ is a best rank-one approximation. From~\eqref{lemma:normal form} we then deduce that $a_{112} = a_{121} = a_{211} = 0$. The only two remaining options for the slices of $A$ are
\begin{equation}
  A= \left(
  \begin{matrix}
    1 & 0 \\
    0 & \pm 1
  \end{matrix}\ \Big\vert\ 
        \begin{matrix}
          0 & \pm 1\\
          \pm 1 & 0
        \end{matrix}
\right).
\end{equation}
But the case $a_{122}=a_{221}=a_{212}=+1$ is also not possible, since it corresponds to the form $p_A(x) = x_1^3+3x_1x_2^2$ whose maximum on the sphere is $\| p_A \|_\infty = \sqrt{2} > 1$. Therefore, $a_{122} = a_{221} = a_{212} = -1$ and $p_A = x_1^3-3x_1x_2^2$ is the cubic Chebyshev form.

We proceed with the induction step. If $A\in \txt{Sym}^{d+1}(\K{R}^2)$ satisfies~\eqref{eq:optimal values} and~\eqref{eq:conditions1}, then its two slices $A_1 = (a_{i_1\dots i_{d} 1})$ and $A_2 = (a_{i_1\dots i_{d}2})$ are extremal $2^d$-tensors. Since $p_{A_1}(e_1) = p_A(e_1) = 1$, it follows from the induction hypothesis that $A_1 = \Ch_{d,2}$. So its entries are given by~\eqref{eq: tensor for Chebyshev}. Let $a_{i_1 \dots i_{d} 2}$ denote an entry of the second slice. Due to the symmetry of $A$, every entry in the second slice, except for the entry $a_{2 \dots 2}$, equals an entry in the first slice after a permutation of the indices. Since this permutation does not affect the number of occurrences of the value $2$ among the indices, the definition~\eqref{eq: tensor for Chebyshev} applies to all these entries as well. It remains to show that the entry $a_{2 \dots 2}$ satisfies~\eqref{eq: tensor for Chebyshev}, that is, equals zero in case $d + 1$ is odd, and equals $(-1)^m$ in case $d+1 = 2m$ is even. This entry is part of the symmetric subtensor $A' = (a_{i_1 i_2 i_3 2 \dots 2})$, which as we have noted above must be extremal as well. Since the entries of the first slice $A_1$ are given by~\eqref{eq: tensor for Chebyshev}, we find that
\[
p_{A'}(x) = (-1)^{m-1}(x_1^3-3x_1x_2^2) + a_{2\cdots 2} x_2^3
\]
if $d+1=2m+1$ is odd. Since $A'$ is extremal, it then follows from the base case $d=3$ that $a_{2\cdots 2} = 0$. In case $d+1 = 2m$ is even, we get
\begin{equation}
  p_{A^{\prime}}(x_1,x_2) = (-1)^{m-1}3x_1^2x_2 + a_{2\cdots 2} x_2^3,
\end{equation}
which by a small consideration implies $a_{2\cdots 2} = (-1)^{m}$. This concludes the proof.
\end{proofof}

\subsection{Ternary cubic tensors}\label{sub:ternary cubics}

In this section we prove Theorem~\ref{thm:max orthogonal rank}. It has been mentioned in section~\ref{subsec: best rank-one ratio etc} how Corollary~\ref{cor:App for ternary cubics} follows from it, and that the statement of Theorem~\ref{thm:ternary cubics} is included in the latter.

The proof of Theorem~\ref{thm:max orthogonal rank} requires a  fact from~\cite{KM2015}. Since it is not explicitly formulated there, we state it here as a lemma.
\begin{lemma}\label{lemma 1}
  For odd $n$ let $A_1, A_2\in \K{R}^n\otimes \K{R}^n$ be two $n\times n$ matrices. If at least one of them is invertible, then there exist orthogonal matrices $\rho, \rho^{\prime} \in O(n)$ such that
  \begin{equation}
    \rho A_1 \rho^{\prime} = 
    \begin{pmatrix}
      B_1 & c_1 \\
      0 & d_1
    \end{pmatrix},\quad \rho A_2 \rho^{\prime} = 
          \begin{pmatrix}
            B_2 & c_2 \\
            0 & d_2
          \end{pmatrix},
  \end{equation}
where $B_1, B_2$ are matrices of size $(n-1)\times (n-1)$, $c_1, c_2$ are $(n-1)$-dimensional vectors and $d_1, d_2$ are real numbers.
\end{lemma}
\begin{proof}
We can assume $A_1$ is invertible. Since $n$ is odd, the matrix $A_1^{-1} A_2^{}$ has at least one real eigenvalue $d$. Then there exists an invertible matrix $P$ such that
\begin{equation}
P^{-1} A_1^{-1} A_2^{} P = \begin{pmatrix} B & c \\ 0 & d \end{pmatrix},
\end{equation}
where $B$ is a matrix of size $(n-1)\times(n-1)$ and $c$ is an $(n-1)$-dimensional vector.
Consider QR decompositions of $A_1 P$ and $P$, that is,
\[
A_1 P = Q_1 R_1, \quad P = Q_2 R_2,
\]
where $Q_1,Q_2$ are orthogonal, and $R_1,R_2$ are upper triangular and invertible. We set $\rho = Q_1^{-1}$ and $\rho^{\prime} = Q_2$. Then
\[
\rho A_1 \rho^{\prime} = R_1^{} P^{-1} A_1^{-1} A_1^{} P R_2^{-1} = R_1^{} R_2^{-1}
\]
is the product of two upper block triangular matrices, hence upper block triangular. Similarly,
\[
\rho A_2^{} \rho^{\prime} = R_1^{} P^{-1} A_1^{-1} A_2^{} P R_2^{-1} = R_1^{} \begin{pmatrix} B & c \\ 0 & d \end{pmatrix} R_2^{-1}
\]
has the asserted upper block triangular structure.
\end{proof}

In~\cite{KM2015} the previous lemma is used to show that for odd $n$ the maximum possible orthogonal rank of an $(n,n,2)$-tensor is $2n-1$. We will only need that the orthogonal rank of a $(3,3,2)$-tensor is not larger than $5$, which actually follows quite easily from the lemma by applying it to the slices.

%

\begin{proofof}{Theorem \ref{thm:max orthogonal rank}}
Note that the aforementioned result of~\cite{KuehnPeetre2006} that $1/\sqrt{7}$ is an upper bound for $\App(\K{R}^3\otimes \K{R}^3\otimes\K{R}^3)$ in combination with~\eqref{eq:lower bound ortho rank} implies that the maxmimal orthogonal rank cannot be less than seven. We will show that it is at most seven.

For $A\in \K{R}^3\otimes \K{R}^3\otimes\K{R}^3$, it is convenient to write $A=(A_1|A_2|A_3)$, where $A_1, A_2, A_3$ are the $3 \times 3$ slices along the third dimension.
 If none of the matrices $A_1,A_2,A_3$ is invertible, each of them can be decomposed into a sum of two rank-one matrices that are orthogonal in the Frobenius inner product: $A_i = u^{(1)}_i\otimes u^{(2)}_i+v^{(1)}_i\otimes v^{(2)}_i$, $i=1,2,3$. This leads to a decomposition of $A$ into at most six pairwise orthogonal rank-one tensors: 
\begin{equation}
  A=\sum\limits_{i=1}^3 u^{(1)}_i\otimes u^{(2)}_i\otimes e_i + v^{(1)}_i\otimes v^{(2)}_i\otimes e_i.
\end{equation}
Assume without loss of generality that the first slice $A_1$ is invertible. Lemma \ref{lemma 1} together with the invariance of orthogonal rank under orthogonal transformations \eqref{eq:action} allows us to assume that $A$ has the form
\begin{align}
 A &= \left( \begin{matrix} *&*&* \\ *&*&* \\ 0&0&* \end{matrix} \ \Bigg\vert \ \begin{matrix} *&*&* \\ *&*&* \\ 0&0&* \end{matrix} \ \Bigg\vert \  \begin{matrix} *&*&* \\ *&*&* \\ *&*&* \end{matrix}  \right) \\
 &= \left( \begin{matrix} *&*&* \\ *&*&* \\ 0&0&0 \end{matrix} \ \Bigg\vert \ \begin{matrix} *&*&* \\ *&*&* \\ 0&0&0 \end{matrix} \ \Bigg\vert \  \begin{matrix} *&*&* \\ *&*&* \\ 0&0&0 \end{matrix}  \right) + \left( \begin{matrix} 0&0&0 \\ 0&0&0 \\ 0&0&* \end{matrix} \ \Bigg\vert \ \begin{matrix} 0&0&0 \\ 0&0&0 \\ 0&0&* \end{matrix} \ \Bigg\vert \  \begin{matrix} 0&0&0 \\ 0&0&0 \\ *&*&* \end{matrix}  \right).
\end{align}
The first term is essentially a $(2,3,3)$-tensor, so its orthogonal rank is at most five by the result of~\cite{KM2015}. In particular, it has a decomposition into at most five pairwise orthogonal rank-one tensors with zero bottom rows. Since the bottom row of the second term is a rank-two matrix the orthogonal rank of $A$ is at most seven.
\end{proofof}

\subsection{On symmetric orthogonal tensors}

We prove Proposition~\ref{propo:n=4,8} below. For the general definition of orthogonal tensors of arbitrary size we refer to~\cite{LNSU2018}. For $n^d$-tensors we can use the recursive definition that $A \in \otimes^d_{j=1}\K{R}^{n}$ is orthogonal if $A \times_j u$ is orthogonal for every $j=1,\dots,d$ and every unit norm vector $u \in \K{R}^n$, where for $d=2$ we agree to the standard definition of an orthogonal matrix. Here and in the proof below we use standard notation $A \times_j u = \left( \sum_{i_j=1}^{n} a_{i_1 \dots i_j \dots i_d} u_{i_j} \right)$ for partial contraction of a tensor $A$ with a vector $u$ along mode $j$, resulting in a tensor of order $d-1$. Note that the above definition implies that every $n^{d'}$-subtensor, $d' < d$, of $A$ is itself orthogonal.

 
\begin{proofof}{Proposition~\ref{propo:n=4,8} and Corollary~\ref{cor:n=4,8}}
 
It has been shown in~\cite{LNSU2018} that an $n^d$-tensor $A$ is orthogonal if and only if it satisfies $\| A \|_2 =1$ and $\| A \|_F = \sqrt{n^{d-1}}$, and such tensors only exist when $n=1,2,4,8$. Therefore, the statement that for $n=2$ the only symmetric orthogonal tensors are the ones obtained from the Chebyshev form $\Ch_{d,2}$ is hence equivalent to Theorem~\ref{thm:binary}. Also, Corollary~\ref{cor:n=4,8} is immediate from Proposition~\ref{propo:n=4,8}.

We thus only have to show that for $n=4,8$ an orthogonal $n^d$-tensor cannot be symmetric. We only consider the case $n=4$; the arguments for $n=8$ are analogous. Since $n^{d'}$-subtensors of an orthogonal tensor are necessarily orthogonal, it is enough to show that orthogonal $4 \times 4 \times 4$ tensors cannot be symmetric. Assume to the contrary that such a tensor $A$ exists. Then $\| A \|_2 = 1$ and $A$ admits a symmetric best rank-one approximation of Frobenius norm one. Since orthogonality and symmetry are preserved under the action of $O(4)$ we can assume that $e_1 \otimes e_1 \otimes e_1$ is the best rank-one approximation of $A$, that is, $a_{111} = \langle A, e_1 \otimes e_1 \otimes e_1 \rangle_F = \| A \|_2 = 1$. On the other hand, the first frontal slice $A \times_3 e_1$ must be a symmetric orthogonal matrix, so it is of the form
\begin{equation}
  A\times_3 e_1 =
  \begin{pmatrix}
  1 & 0 \\
  0 & B
  \end{pmatrix},
\end{equation}
where $B$ is a symmetric orthgonal $3 \times 3$ matrix. By applying further orthogonal transformation that fix the vector $e_1$, we can assume that $B$ is a diagonal matrix with diagonal entries $\varepsilon_1, \varepsilon_2, \varepsilon_3 \in \{+1,-1\}$. 
Since $A$ is symmetric and in fact every slice has to be an orthogonal matrix, we find that $A=(A\times_3 e_1|A\times_3e_2| A\times_3 e_3|A\times_3e_4)$ is of the form
\begin{equation}
  A = 
  \left(
\begin{array}{c|c|c|c}
\begin{matrix}
    1 & 0 & 0 & 0 \\
    0 & \varepsilon_1 & 0 & 0\\
    0 & 0 & \varepsilon_2 & 0 \\
    0 & 0 & 0 & \varepsilon_3 
\end{matrix}
&
\begin{matrix}
0 & \varepsilon_1 & 0 & 0 \\
\varepsilon_1 & 0 & 0 & 0 \\
0 & 0 & 0 & \varepsilon_0 \\
0 & 0 & \varepsilon_0 & 0
\end{matrix}
&
\begin{matrix}
0 & 0 & \varepsilon_3 & 0\\
0 & 0 & 0 & \varepsilon_0 \\
\varepsilon_3 & 0 & 0 & 0 \\
0 & \varepsilon_0 & 0 & 0                         
\end{matrix}
&
\begin{matrix}
0 & 0 & 0 & \varepsilon_4\\
0 & 0 & \varepsilon_0 & 0\\
0 & \varepsilon_0 & 0 & 0\\
\varepsilon_4 & 0 & 0 & 0
\end{matrix}
\end{array}
\right),
\end{equation}
where also $\varepsilon_0\in \{+1,-1\}$. For $i=2,3,4$ the matrices $A\times_3(e_1+e_i)/\sqrt{2}$ must be orthogonal as well, which yields $\varepsilon_0 = 1$ and $\varepsilon_2=\varepsilon_3=\varepsilon_4 = -1$. But then the matrix 
\begin{equation}
A\times_3 \left(\frac{e_1-e_2}{\sqrt{2}} \right) 
=
  \frac{1}{\sqrt{2}} \left(\begin{matrix}
   1 & 1 & 0 & 0 \\
   1 & -1 & 0 & 0 \\
   0 & 0 & -1 & -1 \\
   0 & 0 & -1 & -1 
  \end{matrix}
\right)
\end{equation}
is not orthogonal, which contradicts the assumption that $A$ is an orthogonal tensor. 
\end{proofof}

 
\subsection{Variational characterization}\label{sub:varchar} 
In Theorems~\ref{thm:varchar} and \ref{thm:varcharsym} we characterize critical tensors in $\otimes^d_{j=1}\K{R}^{n_j}$ and $\txt{Sym}^d(\K{R}^n)$ in terms of decompositions into best rank-one approximations. We now prove these results and then derive Corollary \ref{cor:manybest}. Afterwards, we prove Theorem \ref{thm:criteria}.
\begin{proofof}{Theorems~\ref{thm:varchar} and~\ref{thm:varcharsym}}
From section~\ref{sub:gengrad}, specifically~\eqref{eq:nec condition on sphere}, it follows that a nonzero tensor $A^{\prime}\in \otimes^d_{j=1}\K{R}^{n_j}$ is critical in the sense of Definition~\ref{defi:critical} if the tensor $A = A^\prime/\Vert A^\prime\Vert_F$ of Frobenius norm one satisfies
\begin{equation}\label{eq:nec condition spectral norm}
 \mu A \in \partial \| A \|_2
\end{equation}
for some $\mu \in \K{R}$. By~\eqref{eq:spectral norm}, the spectral norm is a max function of the type~\eqref{eq:max function} which is easily shown to satisfy the conditions of~\cite[Thm.~$2.1$]{Clarke1975}. Therefore, its generalized derivative is given by the formula~\eqref{eq:general formula for subdifferential}, which in the case of the max function~\eqref{eq:spectral norm} reads
\begin{equation}\label{eq:subdifferential of spectral norm}
  \partial\Vert A \Vert_2 = \txt{conv}\left\{X : \| X \|_F = 1, \rk(X)=1, \langle A, X\rangle_F=\Vert A\Vert_2\right\},
\end{equation}
where $\text{conv}$ denotes the convex hull.
This lets us write~\eqref{eq:nec condition spectral norm} as
\begin{equation}\label{eq:preexpansion}
  \mu A =  \sum_{\ell=1}^r \alpha_\ell X_\ell,
\end{equation}
where $r>0$ is a natural number,\footnote{By the classical Carath\'eodory theorem one can take $r\leq \dim (\otimes^d_{j=1}\K{R}^{n_j}) +1=n_1\cdots n_d+1$.} $\alpha_1,\dots,\alpha_r>0$ are such that $\alpha_1+\dots+\alpha_r=1$, and $X_\ell$ are rank-one tensors of unit Frobenius norm satisfying $\langle A,X_\ell\rangle_F=\Vert A\Vert_2$. By taking the Frobenius inner product with $A$ itself in~\eqref{eq:preexpansion}, we find that
\[
\mu = \frac{\| A \|_2}{\| A \|_F^2}.
\]
Therefore, after multiplying the resulting equation~\eqref{eq:preexpansion} by $\| A \|_2$ we obtain the asserted statement of Theorem~\ref{thm:varchar}, since, by~\eqref{eq:app}, the rank-one tensors $Y_\ell = \| A \|_2 X_\ell$ are best rank-one approximations of $A$.
%

Considering symmetric tensors instead of general ones in the previous arguments yields a proof of Theorem~\ref{thm:varcharsym}. Here it is crucial that in the definition~\eqref{eq:spectral norm} of spectral norm for symmetric tensors one can restrict to take the maximum over symmetric rank-one tensors of unit Frobenius norm thanks to Banach's theorem; cf.~\eqref{eq:spectral=uniform}.
\end{proofof}
\begin{proofof}{Corollary~\ref{cor:manybest}}
By Proposition \ref{prop:rank} any extremal tensor in $\otimes^d_{j=1}\K{R}^n$ or $\txt{Sym}^d(\K{R}^n)$ must be of rank (respectively, symmetric rank) at least $n$. In particular, there cannot be less than $n$ best rank-one approximations in the expansions \eqref{eq:expansion2} and \eqref{eq:expansionsym2}.
\end{proofof}

\begin{proofof}{Theorem~\ref{thm:criteria}}
Let a tensor $A$ (either in $\otimes_{j=1}^d\K{R}^{n_j}$ or in $\txt{Sym}^d(\K{R}^n)$) be critical, that is, by Theorem~\ref{thm:varchar}, respectively, Theorem~\ref{thm:varcharsym}, 
\begin{equation}\label{eq: decomposition in rank one}
A = \left(\frac{\,\Vert A\Vert_F}{\Vert A\Vert_2}\right)^2 \sum_{\ell = 1}^r \alpha_\ell Y_\ell
\end{equation}
for some (symmetric, if $A$ is symmetric) best rank-one approximations $Y_1,\dots, Y_r$ to $A$, and coefficients $\alpha_1,\dots,\alpha_r>0$ that sum up to one. Recall from section \ref{sub:tensors} that the nuclear norm is dual to the spectral norm. By~\eqref{eq:duality}, this in particular means there exists a tensor $A^*$ satisfying $\| A^* \|_2 \le 1$ and $\| A \|_* = \langle A,A^* \rangle_F$. Note that we then have $\langle X, A^* \rangle_F \le \| X \|_F \| A^* \|_2 \le \| X \|_F$ for every rank-one tensor $X$. Since $\| Y_\ell \|_F = \| A \|_2$, it hence follows from~\eqref{eq: decomposition in rank one} that 
\begin{equation}\label{eq: estimate}
\Vert A \Vert_* = \langle A, A^* \rangle_F = \left(\frac{\,\Vert A\Vert_F}{\Vert A\Vert_2}\right)^2 \sum_{\ell=1}^r \alpha_\ell  \langle Y_\ell, A^* \rangle_F 
\leq\frac{\,\Vert A\Vert^2_F}{\Vert A\Vert_2},  
\end{equation}
which is the converse inequality to \eqref{eq:F2*}. This shows that (i) implies (ii).

Assume now that (ii) holds for a nonzero tensor $A$, that is, $\Vert A\Vert_2\Vert A\Vert_*=\Vert A\Vert_F^2$. By the definition of the nuclear norm there exist $r\in \K{N}$, positive numbers $\beta_1,\dots, \beta_r>0$, and rank-one tensors $X_1,\dots,X_r$ of unit Frobenius norm such that 
\begin{equation}\label{eq:dec5}
  A=\sum_{\ell=1}^r \beta_\ell X_\ell \quad \txt{and}\quad \Vert A\Vert_*=\sum_{\ell=1}^r \beta_\ell.
\end{equation}
If $A$ is symmetric, $X_1,\dots,X_r$ can be taken symmetric~\cite{FriedlandLim}. Taking the Frobenius inner product with $A$ in the first of these equations gives
\begin{equation}\label{eq:ineq3}
\Vert A\Vert_2\Vert A\Vert_*=  \langle A,A\rangle_F= \sum_{\ell=1}^r \beta_\ell\langle A, X_\ell\rangle_F.
\end{equation}
Since $\langle A, X_\ell\rangle_F \le \| A \|_2$ for $\ell=1,\dots, r$ and since $\beta_1,\dots,\beta_r$ sum up to $\| A \|_*$, this equality can hold only if $\langle A, X_\ell\rangle_F = \| A \|_2$ for $\ell=1,\dots, r$. Since, by~\eqref{eq:app}, the rank-one tensors $Y_\ell = \| A \|_2 X_\ell$, $\ell=1,\dots, r$, are then best rank-one approximations of $A$, we see that~\eqref{eq:dec5} is equivalent to~\eqref{eq: decomposition in rank one}, which by Theorems~\ref{thm:varchar} and~\ref{thm:varcharsym} means that $A$ is critical.
\end{proofof}
\begin{remark}
Observe from the proof that decomposition \eqref{eq: decomposition in rank one} of a critical tensor into its best rank-one approximations is also its nuclear decomposition. Vice versa, any nuclear decomposition of a tensor $A$ satisfying $\Vert A\Vert_2\Vert A\Vert_*=\Vert A\Vert_F^2$ can be turned into a convex linear combination of best rank-one approximations of the rescaled tensor $\Vert A\Vert_2^2/\Vert A\Vert_F^2\, A$.
\end{remark}

\subsection{Decomposition of Chebyshev forms}\label{sec:decomposition}

In this section we give the proof of Proposition \ref{prop:decbinary}, which realizes the decomposition of critical tensors into symmetric best rank-one approximations, that is, corresponding powers of linear forms, for the Chebyshev forms $\Ch_{d,2}$.

\begin{proofof}{Proposition~\ref{prop:decbinary}}
Recall that for any $k=0,\dots, d-1$ we denote $\theta_k = \pi k /d$ and $a_k = \cos(\theta_k)$, $b_k = \sin(\theta_k)$. Let us observe that for any such $k$ we can write
\begin{equation}
  \cos(d\theta)  = \txt{Re} ((-1)^ke^{id(\theta-\theta_k)}) = (-1)^k \sum\limits_{\ell=0}^{[d/2]} {d \choose 2\ell} (-1)^\ell \cos(\theta-\theta_k)^{d-2\ell}\sin(\theta-\theta_k)^{2\ell} 
\end{equation}
and therefore
\begin{equation}
  \cos(d\theta) = \frac{1}{d}  \sum\limits_{\ell=0}^{[d/2]}{d\choose 2\ell}(-1)^\ell \sum\limits_{k=0}^{d-1} (-1)^k \cos(\theta-\theta_k)^{d-2\ell}\sin(\theta-\theta_k)^{2\ell}.
\end{equation}
Below we show that for any $\ell=0,\dots,[d/2]$ it holds that
\begin{equation}\label{eq:to prove}
(-1)^\ell \sum\limits_{k=0}^{d-1} (-1)^k \cos(\theta-\theta_k)^{d-2\ell}\sin(\theta-\theta_k)^{2\ell} = \sum\limits_{k=0}^{d-1} (-1)^k \cos(\theta-\theta_k)^d .  
\end{equation}
This together with the identity $\sum_{\ell=0}^{[d/2]} {d\choose 2\ell} = 2^{d-1}$ implies \eqref{eq:dec2} (and hence also \eqref{eq:dec1}).

To derive \eqref{eq:to prove} we write
\begin{align}
{}&{\phantom{{}={}}}(-1)^\ell \sum\limits_{k=0}^{d-1} (-1)^k \cos(\theta-\theta_k)^{d-2\ell}\sin(\theta-\theta_k)^{2\ell}\\
&=\sum\limits_{k=0}^{d-1} (-1)^k \cos(\theta-\theta_k)^{d-2\ell} \sum_{j=0}^\ell {\ell \choose j} \cos(\theta-\theta_k)^{2j} (-1)^{\ell-j}\\
&=\sum\limits_{j=0}^\ell {\ell \choose j}(-1)^{\ell-j} \sum\limits_{k=0}^{d-1} (-1)^k \cos(\theta-\theta_k)^{d-2(\ell-j)}
\end{align}
and claim that for $j=0,\dots,\ell-1$ the inner sum in the last formula is zero. In fact, we will show that for $s=1,\dots,[d/2]$
\begin{equation}\label{eq:to prove2}
  \sum\limits_{k=0}^{d-1} (-1)^k \cos(\theta-\theta_k)^{d-2s} = 0.
\end{equation}
For this let us observe first that Chebyshev polynomials of the first kind $T_{d-2j}(\cos \theta) = \cos((d-2j)\theta)$, $j=1,\dots,[d/2]$, form a basis in the space spanned by univariate real polynomials of degrees $d-2, d-4,\dots,d-2[d/2]$. As a consequence one can express $\cos(\theta-\theta_k)^{d-2s}$ in terms of $T_{d-2j}(\cos(\theta-\theta_k))$ for $j=s,\dots,[d/2]$, and thus in order to prove \eqref{eq:to prove2}, it is enough to show that for $s=1,\dots,[d/2]$ we have
\begin{equation}
  \sum_{k=0}^{d-1} (-1)^k \cos((d-2s)(\theta-\theta_k)) = 0.
\end{equation}
But this follows  from the identity 
\begin{equation}
 \sum_{k=0}^{d-1} (-1)^ke^{i(d-2s)(\theta-\theta_k)}= e^{i(d-2s)\theta} \sum\limits_{k=0}^{d-1} \left(e^{i 2\pi s / d}\right)^k = 0, 
  \end{equation}
  hence the proof is complete
\end{proofof}
We now derive Corollary \ref{cor:deccubic} which, in particular, implies that the cubic Chebyshev forms $\Ch_{3,n}$ are critical for the ratio $\Vert p\Vert_\infty/\Vert p \Vert_B, p\in P_{3,n}$.
\begin{proofof}{Corollary \ref{cor:deccubic}}
From Proposition \ref{prop:decbinary} we get
\begin{equation}\label{eq:Ch32}
  \Ch_{3,2}(x_1,x_2) = x_1^3-3x_1x^2_2 = \frac{4}{3}\left(x_1^3-\left(\frac{x_1-\sqrt{3}x_2}{2}\right)^3+\left(\frac{-x_1+\sqrt{3}x_2}{2}\right)^3\right).
\end{equation}
We then write
\begin{align}
  \Ch_{3,n}(x) &= x_1^3-3x_1(x_2^2+\dots+x_n^2) =  - (n-2)x_1^3+\sum_{i=2}^n \left(x_1^3-3x_1x^2_i\right) \\
&= - (n-2)x_1^3 +\frac{4}{3}\sum_{i=2}^n x_1^3-\left(\frac{x_1-\sqrt{3}x_i}{2}\right)^3+\left(\frac{-x_1+\sqrt{3}x_i}{2}\right)^3, 
\end{align}
where we applied \eqref{eq:Ch32} to each binary Chebyshev form $\Ch_{3,2}(x_1,x_i)=x_1^3-3x_1x_i^2$. The obtained formula is equivalent to the asserted one \eqref{eq:Ch3n}.
\end{proofof}

\subsection{Local minimality of cubic Chebyshev forms}\label{sub:localoptimality}

This subsection is devoted to the proof of Theorem~\ref{thm:locmin}, which states that the cubic Chebyshev form $\Ch_{3,n}(x)=x_1^3-3x_1(x_2^2+\dots+x_n^2)$ is a local minimum for the ratio of uniform and Bombieri norms. We denote by $G\simeq O(n-1)\subset O(n)$ the subgroup consisting of orthogonal transformations that preserve the point $(1,0,\dots,0)\in \K{R}^n$. Note that $G\subset O(n)$ is of codimension $n-1$ and that $\Ch_{3,n}$ is invariant under $G$. In particular, the $O(n)$-orbit of $\Ch_{3,n}$ is at most $(n-1)$-dimensional. In the following lemma we describe the tangent space to this orbit, a result that we need for the proof of Theorem~\ref{thm:locmin}.
\begin{lemma}\label{lem:orbit}
The $O(n)$-orbit of $\Ch_{3,n}$ has dimension $n-1$ and its tangent space at $\Ch_{3,n}$ consists of all reducible cubics of the form $\ell \cdot q$, where $\ell$ is a linear form that vanishes at $(1,0,\dots,0)\in \K{R}^n$ and $q(x) = 3x_1^2-x_2^2-\dots-x_n^2$.
\end{lemma}

\begin{proof}
For $i=2,\dots,n$ let us consider the elementary rotation $R_i(\varphi)\in O(n)$ in the $(x_1,x_i)$-plane, that is, $R_i(\varphi)$ is given by the $n \times n$ matrix whose only non-zero entries are $(R_i(\varphi))_{11}=(R_i(\varphi))_{ii}=\cos(\varphi)$, $(R_i(\varphi))_{1i}=-(R_i(\varphi))_{i1}=\sin \varphi$, and $(R_i(\varphi))_{jj}=1$ for $j \neq 1, i$. It is a straightforward calculation to check that the tangent vector to the curve $\varphi \mapsto R_i(\varphi)^*\Ch_{3,n}$ at $\varphi=0$ is a nonzero cubic proportional to $x_iq$. For $i=2,\dots,n$ these $n-1$ tangent vectors are linearly independent, and, since the $O(n)$-orbit of $\Ch_{3,n}$ is at most $(n-1)$-dimensional, the claim follows.
\end{proof}

\begin{proofof}{Theorem \ref{thm:locmin}}
Let $S=\{p\in P_{3,n}:\Vert p\Vert_B=\Vert \Ch_{3,n}\Vert_B\}$ denote the sphere of radius $\| \Ch_{3,n} \|_B$ in $(P_{3,n}, \Vert\cdot\Vert_B)$. Denote by $H$ any $(n-1)$-dimensional submanifold of $O(n)$ that passes through the identity $\txt{id}\in O(n)$ and intersects $G$ transversally at $\txt{id}\in H\cap G$. Denote also by $M$ any submanifold of $S$ that has codimension $n-1$, passes through $\Ch_{3,n}\in S$, and intersects the $O(n)$-orbit of $\Ch_{3,n}$ transversally at $\Ch_{3,n}$. Consider now the smooth map $f: H\times M\rightarrow S$, $(h,m)\mapsto h^*m$ and note that by construction the differential of $f$ at $(\txt{id},\Ch_{3,n})$ is surjective. In particular, $f$ maps some open neighborhood of $(\txt{id},\Ch_{3,n})\in H\times M$ to an open neighborhood of $\Ch_{3,n}\in S$. Therefore, since the uniform norm is $O(n)$-invariant, in order to prove the claim of the theorem it is enough to show that $\Ch_{3,n}\in M$ is a local minimum of the uniform norm restricted to $M$. To prove the latter let us denote by $T$ the sphere of radius $\| \Ch_{3,n} \|_B$ in the tangent space to $M$ at $\Ch_{3,n}$. We claim that there exists a constant $\delta > 0$ such that for any $p \in T$ there exists a point $x\in \K{R}^n$, $\Vert x\Vert=1$, 
such that
\begin{equation}\label{eq: claim for p}
\abs{\Ch_{3,n}(x)} = 1 \quad \text{and} \quad \Ch_{3,n}(x) p(x) \ge \delta. 
\end{equation}
It then follows for the geodesic $\gamma_p(t)=\cos t \cdot \Ch_{3,n}+\sin t \cdot p$ that
\[
 \| \gamma_p(t) \|_\infty \ge \abs{\cos t\, \Ch_{3,n}(x) + \sin t \,p(x)} \ge \cos t + \delta \sin t \geq 1 = \| \Ch_{3,n} \|_\infty
\]
for all $0 \le t \le t_\delta$, where $t_\delta>0$ depends only on $\delta$. This proves that $\Ch_{3,n}\in M$ is a local minimum of the uniform norm restricted to $M$.

In order to show~\eqref{eq: claim for p} let us define
\begin{align}
  C_n &= \{\pm e_1 \}\cup\left\{\pm \frac{1}{2}e_1+\frac{\sqrt{3}}{2} \rho e_2 : \rho\in G\right\} \\
   &=\{\pm e_1\}\cup \{x\in \K{R}^n: \| x \| = 1, \, 3x_1^2-x_2^2-\dots-x_n^2=0\},
\end{align}
where $e_1=(1,0,\dots,0)$ and $e_2=(0,1,0,\dots,0)$. From the $G$-invariance of $\Ch_{3,n}$ one can see that $C_n$ is the set of unit vectors $x\in \K{R}^n$, $\Vert x\Vert=1$, satisfying $|\Ch_{3,n}(x)|=1$ and 
\begin{equation}\label{eq:values}
  \Ch_{3,n}(\pm e_1)=\pm 1,\quad \Ch_{3,n}\left(\pm\frac{1}{2}e_1+\frac{\sqrt{3}}{2}\rho e_2\right)=\mp 1,\quad \rho \in G.
\end{equation}
From Lemma~\ref{lem:orbit} it follows that a nonzero form $p \in P_{3,n}$ vanishes on $C_n$ if and only if it belongs to the tangent space of the $O(n)$-orbit of $\Ch_{3,n}$ at $\Ch_{3,n}$.\footnote{It is interesting to state this in the language of symmetric tensors: the tangent space of the $O(n)$-orbit of the symmetric tensor associated with $\Ch_{3,n}$ is the orthogonal complement of the span of its symmetric best rank-one approximations.} In particular, no $p\in T$ vanishes on the whole of $C_n$. From compactness of both $C_{n}$ and $T$ we hence conclude that
\begin{equation}\label{eq: max for abs p}
 \max_{x \in C_n} \abs{p(x)} \ge \delta^\prime
\end{equation}
for some $\delta^\prime > 0$ and all $p \in T$. Now put $\delta= \delta^{\prime}/(10n)$. Given $p\in T$, let $x^\prime\in C_n$ be such that $|p(x^\prime)|\geq \delta^\prime$. If $p(x^\prime)$ and $\Ch_{3,n}(x^\prime)$ have the same sign, \eqref{eq: claim for p} obviously holds as $\delta^\prime > \delta$. We now treat the case when $\Ch_{3,n}(x^\prime)p(x^\prime)<0$. Note first that, as a consequence of the $G$-invariance of $\Ch_{3,n}$, together with the decomposition \eqref{eq:Ch3n}, we have the whole family of decompositions
\begin{equation}\label{eq:manydec}
   \Ch_{3,n}(x) = \frac{n+2}{3}x_1^3+\frac{4}{3}\sum\limits_{i=2}^n -\left\langle v^+_{\rho,i}, x\right\rangle^3+\left\langle v_{\rho,i}^-, x\right\rangle^3,\quad \rho\in G,
\end{equation}
where $v^\pm_{\rho,i}= \pm 1/2 e_1+\sqrt{3}/2 \rho e_i$ and $e_i$ is the $i$th unit vector, $i=1,\dots,n$.
The set of possible $v^\pm_{\rho,i}$ for different $\rho\in G$ coincides with $C_n \setminus \{ \pm e_1\}$. Therefore, $x^\prime$ either is $\pm e_1$ (in which case we can assume that $x^\prime = e_1$) or is among $v^\pm_{\rho,i}$, $i=2,\dots, n$, for some $\rho \in G$. Since $p$ is tangent to $S$ at $\Ch_{3,n}$, that is, $\langle p,\Ch_{3,n}\rangle_B=0$, and since $\langle p,\langle v,\cdot\rangle^3\rangle_B = p(v)$, we get from \eqref{eq:values} and \eqref{eq:manydec} that
\[
   0 = \langle p, \Ch_{3,n} \rangle_B = \frac{n+2}{3}\Ch_{3,n}(e_1)p(e_1)+\frac{4}{3}\sum\limits_{i=2}^n\Ch_{3,n}(v^+_{\rho,i}) p(v^+_{\rho,i})+\Ch_{3,n}(v^-_{\rho,i})p(v^-_{\rho,i}).
\]
One of these terms features $\Ch_{3,n}(x^\prime)p(x^\prime) \le -\delta^\prime$. 
Elementary estimates then show that for some $x$ among $e_1$ and $v^{+}_{\rho,i}$, $v^-_{\rho,i}$, $i=2,\dots, n$, we must have $\Ch_{3,n}(x)p(x)\geq \delta^{\prime}/(10n) = \delta$. We thus have verified~\eqref{eq: claim for p} for all $p \in T$ and some $\delta > 0$, which concludes the proof.
\end{proofof}

\subsection*{Acknowledgment}
We thank Zhening~Li for pointing out a counterexample to the global optimality of Chebyshev forms $\Ch_{3,n}$ as presented in section~\ref{sec: Chebyshev forms}.


\noindent

\bigskip{\footnotesize
\noindent
{International School for Advanced Studies, 34136 Trieste, Italy}\\  
\noindent
\texttt{agrachev@sissa.it}}

\bigskip{\footnotesize
\noindent
{Max Planck Institute for Mathematics in the Sciences, 04103 Leipzig, Germany}\\ 
\noindent
\texttt{kozhasov@mis.mpg.de}}

\bigskip{\footnotesize
\noindent
{Max Planck Institute for Mathematics in the Sciences, 04103 Leipzig, Germany}\\
\noindent  
\texttt{uschmajew@mis.mpg.de}}

\end{document}